\documentclass[leqno,12pt]{amsart} 

\usepackage{amssymb} 
\usepackage{amsmath}
\usepackage{amsthm}
\usepackage{amscd}
\usepackage{bm}
\usepackage{graphicx,psfrag,wrapfig}
\theoremstyle{plain} 
\newtheorem{theorem}{Theorem}[section]  
\newtheorem{lemma}[theorem]{Lemma}
\newtheorem{corollary}[theorem]{Corollary}
\newtheorem{proposition}[theorem]{Proposition}

\newtheorem{sublemma}[theorem]{Sublemma}
\theoremstyle{definition} 
\newtheorem{definition}[theorem]{Definition}
\newtheorem{remark}[theorem]{Remark}

\newcommand{\affine}{\mathbb{C}}

\newcommand{\const}{\mathrm{const}}
\newcommand{\norm}[1]{\left|\!\left|#1\right|\!\right|}
\newcommand{\ad}{\mathrm{ad}}
\newcommand{\supp}{\mathrm{supp}}
\newcommand{\vol}{\mathrm{vol}}

\newcommand{\moduli}{\mathcal{M}}

\newcommand{\conf}{\mathrm{Conf}}

\begin{document}

\title[An open four-manifold having no instanton]
{An open four-manifold having no instanton} 

\author[M. Tsukamoto]{Masaki Tsukamoto}

\date{\today}

\keywords{Yang-Mills theory, instanton, open four-manifold, infinite connected sum}

\subjclass[2000]{53C07}

\maketitle

\begin{abstract}
Taubes proved that all compact oriented four-manifolds admit non-flat instantons.
We show that there exists a non-compact oriented four-manifold having no non-flat instanton.
\end{abstract}


\section{introduction}\label{section: introduction}
Taubes \cite{Taubes} proved that all compact oriented Riemannian $4$-manifolds admit non-flat instantons.
To be precise, if $X$ is a compact oriented Riemannian $4$-manifold then 
there exists a principal $SU(2)$-bundle $E$ on $X$ which admits a non-flat anti-self-dual (ASD) connection.
(Taubes \cite{Taubes} considered self-dual connections.
But recently people usually study anti-self-dual ones.
So we consider anti-self-dual connections in this paper.)
The purpose of this paper is to show that an analogue of this striking existence theorem does \textit{not} hold for 
general non-compact $4$-manifolds.

Let $(\affine P^2)^{\sharp \mathbb{Z}}$ be the connected sum of the infinite copies of the complex projective plane
$\affine P^2$ indexed by integers.
(The precise definition of this infinite connected sum will be given in Section \ref{subsection: construction}.)
$(\affine P^2)^{\sharp \mathbb{Z}}$ is a non-compact oriented $4$-manifold.
\begin{theorem} \label{thm: main theorem in introduction}
There exists a complete Riemannian metric $g$ on $(\affine P^2)^{\sharp\mathbb{Z}}$ satisfying the following.
If $A$ is a $g$-ASD connection on 
a principal $SU(2)$-bundle over $(\affine P^2)^{\sharp \mathbb{Z}}$ satisfying 
\begin{equation} \label{eq: finite energy condition in the introduction}
 \int_X |F_A|_g^2d\vol_g < +\infty,
\end{equation}
then $A$ is flat.
Here $F_A$ is the curvature of $A$. $|\cdot|_g$ and $d\vol_g$ are the norm and the volume form with respect to 
the metric $g$. 
A connection $A$ is said to be $g$-ASD if it satisfies $*_g F_A = -F_A$ where $*_g$ is the Hodge star with respect to $g$.
\end{theorem}
For a more general and precise statement, see Theorem \ref{thm: main theorem}.
As far as I know, this is the first example of oriented Riemannian 
$4$-manifolds which cannot admit any non-flat instanton.
\begin{remark}
I think that the following question is still open:
Is there an oriented Riemannian $4$-manifold which does not have any non-flat ASD connection (not 
necessarily satisfying the finite energy condition (\ref{eq: finite energy condition in the introduction}))?
We studied infinite energy ASD connections and their infinite dimensional moduli spaces
in \cite{Matsuo-Tsukamoto}, \cite{Tsukamoto 1}, \cite{Tsukamoto 2}.
\end{remark}
A naive idea toward the proof of Theorem \ref{thm: main theorem in introduction} is as follows.
Let $g$ be a Riemannian metric on $(\affine P^2)^{\sharp \mathbb{Z}}$.
For each integer $n\geq 0$, let $M(n,g)$ be the moduli space of $SU(2)$ $g$-ASD connections on $(\affine P^2)^{\sharp \mathbb{Z}}$ 
satisfying $\int_{(\affine P^2)^{\sharp \mathbb{Z}}} |F_A|_g^2 d\vol_g =8\pi^2 n$.
We have $b_1((\affine P^2)^{\sharp \mathbb{Z}})=0$ and, formally, $b_+((\affine P^2)^{\sharp \mathbb{Z}}) = +\infty$.
Therefore, if we formally apply the usual virtual dimension formula 
\cite[Section 4.2.5]{Donaldson-Kronheimer} to $M(n,g)$, we get 
\[ \dim M(n,g) = 8n -3(1-b_1((\affine P^2)^{\sharp \mathbb{Z}})+b_+((\affine P^2)^{\sharp \mathbb{Z}})) = 8n -\infty = -\infty.\]
This suggests the following observation:
If we can achieve the transversality of the moduli spaces $M(n,g)$ by choosing the metric $g$ sufficiently generic,
then all $M(n,g)$ $(n\geq 1)$ become empty.
($M(0,g)$ is the moduli space of flat $SU(2)$ connections, and it does not depend on the choice of a Riemannian metric.)

\textbf{Acknowledgement.}
I wish to thank Professor Kenji Fukaya most sincerely for his help and encouragement.
I was supported by Grant-in-Aid for Young Scientists (B) (21740048).

\section{Infinite connected sum} \label{section: infinite connected sum}
\subsection{Construction}\label{subsection: construction}
Let $Y$ be a simply-connected compact oriented $4$-manifold. 
Let $x_1,x_2\in Y$ be two distinct points, and set $\hat{Y}:=Y \setminus \{x_1,x_2\}$.
Choose a Riemannian metric $h$ on $\hat{Y}$ which becomes a tubular metric on the end (i.e. around $x_1$ and $x_2$).
This means that there is a compact set $K\subset \hat{Y}$ such that $\hat{Y} \setminus K =Y_- \sqcup Y_+$ with
$Y_- = (-\infty, -1)\times S^3$ and $Y_+ =(1,+\infty)\times S^3$. 
Here ``$=$'' means that they are isomorphic as oriented Riemannian manifolds.
($S^3 = S^3(1) = \{x\in \mathbb{R}^4|\, |x|=1\}$ is endowed with the Riemannian metric induced by the 
standard Euclidean metric on $\mathbb{R}^4$.)
We can suppose that there is a smooth function $p:\hat{Y} \to \mathbb{R}$ satisfying the following conditions:
$p(K) = [-1,1]$. 
$p$ is equal to the projection to $(-\infty,-1)$ on $Y_-= (-\infty,-1)\times S^3$,
and $p$ is equal to the projection to $(1,+\infty)$ on $Y_+= (1,+\infty)\times S^3$.
For $T>2$, we set $Y_T := p^{-1}(-T+1,T-1) =(-T+1, -1)\times S^3\cup K\cup (1,T-1)\times S^3$.
(Later we will choose $T$ large.)

Let $Y^{(n)}$ be the copies of $Y$ indexed by integers $n\in \mathbb{Z}$.
We denote $K^{(n)}$, $Y_-^{(n)}$, $Y_+^{(n)}$, $p^{(n)}$, $Y^{(n)}_T$ as the copies of
$K$, $Y_-$, $Y_+$, $p$, $Y_T$.
($K^{(n)}, Y^{(n)}_-, Y^{(n)}_+, Y^{(n)}_T \subset Y^{(n)}$ and $p^{(n)}:Y^{(n)}\to \mathbb{R}$.)
We define $X = Y^{\sharp \mathbb{Z}}$ by 
\[ X := \bigsqcup_{n\in \mathbb{Z}} Y_T^{(n)} /\sim ,\]
where we identify $Y^{(n)}_T \cap Y^{(n)}_+$ with $Y^{(n+1)}_T \cap Y^{(n+1)}_-$ by 
\begin{equation} \label{eq: definition of infinite connected sum}
 \begin{split}
 Y^{(n)}_T \cap Y^{(n)}_+ = (1,T-1)\times S^3 &\ni (t, \theta) \\
  &\sim (t-T, \theta)\in  (-T+1, -1) \times S^3
   = Y^{(n+1)}_T \cap Y^{(n+1)}_-.
 \end{split}
\end{equation}
We define $q:X\to \mathbb{R}$ by setting $q(x) := nT +p^{(n)}(x)$ on $Y^{(n)}_T$.
This is compatible with the above identification (\ref{eq: definition of infinite connected sum}).
The identification (\ref{eq: definition of infinite connected sum}) is an orientation preserving isometry.
Hence $X$ has an orientation 
and a Riemannian metric which coincide with the given ones over $Y_T^{(n)}$.
We denote the Riemannian metric on $X$ (given by this procedure) by $g_0$. 
$g_0$ depends on the Riemannian metric $h$ on $Y$ and the parameter $T$.

Since $Y$ is simply-connected, $X$ is also simply-connected.
The homology groups of $X$ are given as follows:
\begin{equation*}
 H_0(X) =\mathbb{Z},\quad H_1(X) = 0, \quad H_2(X) \cong H_2(Y)^{\oplus\mathbb{Z}},  \quad
 H_3(X) = \mathbb{Z},\quad  H_4(X)=0.
\end{equation*}
$H_2(X)$ is of infinite rank if $b_2(Y)\geq 1$.
For every $n\in \mathbb{Z}$, the inclusion $Y^{(n)}_T\cap Y^{(n)}_+ \subset X$ induces an isomorphism 
$H_3(Y^{(n)}_T\cap Y^{(n)}_+) \cong H_3(X)$.
The fundamental class of the cross-section $S^3\subset Y^{(n)}_T\cap Y^{(n)}_+ = (1,T-1)\times S^3$ becomes a generator of $H_3(X)$.

\subsection{Statement of the main theorem}\label{subsection: statement of the main theorem}
Theorem \ref{thm: main theorem in introduction} in Section \ref{section: introduction} follows from the following theorem.
\begin{theorem} \label{thm: main theorem}
Suppose $b_-(Y)=0$ and $b_+(Y)\geq 1$.
If $T$ is sufficiently large, then there exists a complete Riemannian metric $g$ on 
$X = Y^{\sharp \mathbb{Z}}$ satisfying the following conditions (a) and (b).

\noindent 
(a) $g$ is equal to the periodic metric $g_0$ (defined in Section \ref{subsection: construction}) outside a compact set.

\noindent 
(b) If $A$ is a $g$-ASD connection on a principal $SU(2)$ bundle $E$ on $X$ satisfying 
\begin{equation} \label{eq: finite energy condition in the statement of the main theorem}
 \int_X |F_A|^2_g d\vol_g < \infty,
\end{equation}
then $A$ is flat.
\end{theorem}
The proof of this theorem will be given in Section \ref{section: proof of the main theorem}.
\begin{remark}\label{remark: remark on main theorem}
(i) If a Riemannian metric $g$ on $X$ satisfies the condition (a), then it is complete.

\noindent
(ii)
From the condition (a), 
the above (\ref{eq: finite energy condition in the statement of the main theorem}) is
equivalent to 
\[ \int_X |F_A|^2_{g_0} d\vol_{g_0} < \infty.\]
\noindent
(iii) Since $X$ is non-compact, all principal $SU(2)$-bundles on it are isomorphic to the 
product bundle $X\times SU(2)$.
Hence we can assume that the principal $SU(2)$-bundle $E$ in the condition (b)
is equal to the product bundle $X\times SU(2)$.
\end{remark}

\subsection{Ideas of the proof of Theorem \ref{thm: main theorem}} \label{subsection: idea of the proof of main theorem}
In this subsection we explain the ideas of the proof of Theorem \ref{thm: main theorem}.
Here we ignore several technical issues.
Hence the real proof is different from the following argument in many points.

Let $g$ be a Riemannian metric on $X$ which is equal to $g_0$ outside a compact set.
Let $E=X\times SU(2)$ be the product principal $SU(2)$-bundle over $X$.
If a $g$-ASD connection $A$ on $E$ satisfies 
$\int_X |F_A|_g^2d\vol_g < \infty$, then 
we can show that $\frac{1}{8\pi^2}\int_X|F_A|_g^2 d\vol_g$ is a non-negative integer.
For each integer $n\geq 0$, we define 
$M(n,g)$ as the moduli space of $g$-ASD connections $A$ on $E$ satisfying 
$\frac{1}{8\pi^2}\int_X|F_A|_g^2d\vol_g = n$.
Take $[A]\in M(n,g)$.
We want to study a local structure of $M(n, g)$ around $[A]$.

Set $D_A := -d_A^{*}+ d_A^{+_g} :\Omega^1(\ad E)\to (\Omega^0\oplus \Omega^{+_g})(\ad E)$.
Here $d_A^{*}$ is the formal adjoint of $d_A:\Omega^0(\ad E)\to \Omega^1(\ad E)$ with respect to $g_0$, and
$d_A^{+_g}$ is the $g$-self-dual part of $d_A:\Omega^1(\ad E) \to \Omega^2(\ad E)$.
(Indeed we need to use appropriate weighted Sobolev spaces, and the definition of $D_A$ should be 
modified with the weight. But here we ignore these points.)
The equation $d_A^{*}a=0$ for $a\in \Omega^1(\ad E)$ is the Coulomb gauge condition, and the equation 
$d_A^{+_g}a=0$ is the linearization of the ASD equation $F^{+_g}(A+a)=0$. 
Therefore we expect that we can get an information on the local structure of $M(n, g)$ 
from the study of the operator $D_A$.
The most important point of the proof is to show the following three properties of $D_A$.
(In other words, we need to choose an appropriate functional analysis setup in order to establish these 
properties.)

\noindent 
(i) The kernel of $D_A$ is finite dimensional.

\noindent 
(ii) The image of $D_A$ is closed in $(\Omega^0\oplus \Omega^{+_g})(\ad E)$.

\noindent 
(iii) The cokernel of $D_A$ is infinite dimensional.

Then the local model (i.e. the Kuranishi description) of $M(n,g)$ around $[A]$ is given by the zero set of a 
map 
\[ f: \mathrm{Ker}D_A\to \mathrm{Coker} D_A.\]
(Rigorously speaking, the map $f$ is defined only in a small neighborhood of the origin.)
From the conditions (i) and (iii), this is a map from the finite dimensional space to the infinite dimensional one.
Therefore (we can hope that) if we perturb the map $f$ appropriately, then the zero set disappear. 
The parameter $g$ gives sufficient perturbation, and we can prove that $M(n,g)$ is empty for $n\geq 1$ and generic $g$.

\textbf{Organization of the paper:}
In Section \ref{subsection: anti-self-duality and conformal structure}, 
we review the basic facts on anti-self-duality and conformal structure.
In the above arguments we consider the moduli space $M(n,g)$ parametrized by Riemannian metrics $g$.
But ASD equation depends only on conformal structures, and hence technically it is better to parameterize
ASD moduli spaces by conformal structures.
Section \ref{subsection: anti-self-duality and conformal structure} is a preparation for this consideration.
In Section \ref{subsection: Eigenvalues of Laplacians on differential forms over S^3}, we prepare some estimates 
relating to the Laplacians.

In Section \ref{section: decay estimate of instatons}
we study the decay behavior of instantons over $X$, and show that 
they decay ``sufficiently fast''.
This is important in showing that all instantons can be ``captured'' by the functional analysis setups constructed in 
Sections \ref{section: linear theory} and \ref{subsection: metric perturbation}.

Section \ref{section: preliminaries for linear theory} is a preparation for Section 
\ref{section: linear theory}.
In Section \ref{section: linear theory} we study a (modified version of) operator
$D_A = -d_A^{*} + d_A^{+_d}$ and establish the above mentioned properties (i), (ii), (iii).

In Section \ref{section: non-existence of reducible instantons} we show that 
there is no non-flat reducible instantons on $E$.
Here the condition $b_-(Y)=0$ is essentially used.

Sections \ref{subsection: Sard-Smale's theorem} and
\ref{subsection: Review of Floer's functional space}
are preparations for the perturbation argument in Section 
\ref{subsection: metric perturbation}.
In Section \ref{subsection: metric perturbation} we establish a transversality
by using Freed-Uhlenbeck's metric perturbation.
Here we use the results established in Sections \ref{section: linear theory} and \ref{section: non-existence of reducible instantons}.
Combining the results in Sections \ref{section: decay estimate of instatons} and \ref{subsection: metric perturbation}, 
we prove Theorem \ref{thm: main theorem} in Section \ref{section: proof of the main theorem}.

\section{Some preliminaries}  \label{section: some preliminaries}
\subsection{Anti-self-duality and conformal structure} \label{subsection: anti-self-duality and conformal structure}
In this subsection we review some well-known facts on the relation between anti-self-duality and conformal 
structure.
Specialists of the gauge theory don't need to read the details of the arguments in this subsection.
The references are Donaldson-Sullivan \cite[pp. 185-187]{Donaldson-Sullivan} and 
Donaldson-Kronheimer \cite[pp. 7-8]{Donaldson-Kronheimer}.

We start with a linear algebra.
Let $V$ be an oriented real $4$-dimensional linear space.
We fix an inner product $g_0$ on $V$.
The orientation and inner product give a natural isomorphism 
$\Lambda^4(V)\cong \mathbb{R}$, and we define 
a quadratic form $Q: \Lambda^2(V)\times \Lambda^2(V)\to \mathbb{R}$
by $Q(\xi, \eta) := \xi\wedge \eta \in \Lambda^4(V)\cong \mathbb{R}$.
The dimensions of maximal positive subspaces and maximal negative subspaces with respect to $Q$ are both $3$  

Let $g$ and $g'$ be two inner products on $V$.
They are said to be conformally equivalent if there is $c>0$ such that 
$g_2 = cg_1$.
Let $\conf(V)$ be the set of all conformal equivalence classes of inner-products on $V$.
$\conf(V)$ naturally admits a smooth manifold structure.

We define $\conf'(V)$ as the set of all $3$-dimensional subspaces $U\subset \Lambda^2(V)$ satisfying 
$Q(\omega, \omega) <0$ for all non-zero $\omega\in U$.
$\conf'(V)$ depends on the orientation of $V$, but it is independent of the choice of the 
inner product $g_0$.
$\conf'(V)$ is an open set of the Grassmann manifold $Gr_3(\Lambda^2(V))$, and hence it is also 
a smooth manifold.

Let $\Lambda^+$ be the space of $\omega\in \Lambda^2(V)$ which is self-dual with respect to $g_0$,
and $\Lambda^-$ be the space of $\omega\in \Lambda^2(V)$ which is anti-self-dual with respect to $g_0$.
We define $\conf''(V)$ be the set of linear map $\mu:\Lambda^-\to \Lambda^+$ satisfying 
$|\mu|<1$ (i.e. $|\mu(\omega)| < |\omega|$ for all non-zero $\omega\in \Lambda^-$ where 
the norm $|\cdot|$ is defined by $g_0$).
This is also a smooth manifold as an open set of $\mathrm{Hom}(\Lambda^-, \Lambda^+)$.
The map 
\[ \conf''(V) \to \conf'(V), \quad \mu \mapsto \{\omega + \mu(\omega)|\, \omega\in \Lambda^-\} \]
is a diffeomorphism.
Hence $\conf'(V)$ is contractible. (In particular it is connected.)
\begin{lemma}\label{lemma: the description of conformal structures, linear algebra}
The map 
\begin{equation} \label{eq: diffeomorphism between conf(V) and conf'(V)}
 \conf(V)\to \conf'(V), \quad [g]\mapsto 
 \{\omega\in \Lambda^2(V)|\, \text{$\omega$ is anti-self-dual with respect to $g$}\} ,
\end{equation}
is a diffeomorphism.
\end{lemma}
\begin{proof}
For $A\in SL(V)$ and $[g]\in \conf(V)$ we define $[Ag]\in \conf(V)$ by setting 
$(Ag)(u, v) := g(A^{-1}u, A^{-1}v)$.
In this manner $SL(V)$ transitively acts on $\conf(V)$, and the isotropy subgroup at $[g_0]$ is equal to 
$SO(V) = SO(V, g_0)$.
Hence $\conf(V) \cong SL(V)/SO(V)$.
On the other hand, the Lie group $SO(\Lambda^2(V), Q)$ ($\cong SO(3, 3)$) naturally acts on $\conf'(V)$.
This action is transitive.
(For $U\in \conf'(V)$ set $U' := \{\omega\in \Lambda^2(V)|\, Q(\omega, \eta)=0 \>(\forall \eta\in U)\}$.
$\Lambda^2(V) = U\oplus U'$. $Q$ is negative definite on $U$ and positive definite on $U'$.
By choosing orthonormal bases on $U$ and $U'$ with respect to $Q$, we can construct $A\in SO(\Lambda^2(V), Q)$
satisfying $A(\Lambda^-)=U$.)

Let $SO(\Lambda^2(V),Q)_0$ be the identity component of $SO(\Lambda^2(V), Q)$.
Since $\conf'(V)$ is connected, $SO(\Lambda^2(V),Q)_0$ also transitively acts on $\conf'(V)$.
The isotropy group of this action at $\Lambda^-\in \conf'(V)$ is equal to 
$SO(\Lambda^+)\times SO(\Lambda^-)$.
(It is easy to see that if $A\in SO(\Lambda^2(V),Q)_0$ fixes $\Lambda^-$ then it also 
fixes $\Lambda^+$. Hence $A \in O(\Lambda^+)\times O(\Lambda^-)$.
Since $\conf'(V)$ is contractible, the isotropy subgroup
must be connected. Therefore $A \in SO(\Lambda^+)\times SO(\Lambda^-)$.)
Thus $\conf'(V)\cong SO(\Lambda^2(V), Q)_0/ SO(\Lambda^+)\times SO(\Lambda^-)$

$SL(V)$ naturally acts on $\Lambda^2(V)$, and it preserves the quadratic form $Q$.
Hence we have a homomorphism $f:SL(V)\to SO(\Lambda^2(V), Q)_0$.
A direct calculation shows that it induces an isomorphism between their Lie algebras.
Hence the homomorphism $f:SL(V)\to SO(\Lambda^2(V), Q)_0$ is a (surjective) covering map.
$f^{-1}(SO(\Lambda^+)\times SO(\Lambda^-))$ is equal to $SO(V)$.
(It is easy to see that $SO(V)\subset f^{-1}(SO(\Lambda^+)\times SO(\Lambda^-))$ and their dimensions are both $6$.
$SL(V)$ is connected and  
$SL(V)/f^{-1}(SO(\Lambda^+)\times SO(\Lambda^-)) \cong SO(\Lambda^2(V), Q)_0/ SO(\Lambda^+)\times SO(\Lambda^-)
\cong \conf'(V)$ is contractible.
Hence $f^{-1}(SO(\Lambda^+)\times SO(\Lambda^-))$ must be connected.
Therefore it is equal to $SO(V)$.)
Thus $SL(V)/SO(V)\cong SO(\Lambda^2(V), Q)_0/SO(\Lambda^+)\times SO(\Lambda^-)$.
This gives a diffeomorphism $\conf(V)\cong \conf'(V)$, and this diffeomorphism coincides with the above 
map (\ref{eq: diffeomorphism between conf(V) and conf'(V)}).
\end{proof}
Let $M$ be an oriented $4$-manifold (not necessarily compact), and $g_0$
be a smooth Riemannian metric on $M$.
Two Riemannian metrics $g$ and $g'$ on $M$ are said to be conformally equivalent if 
there is a positive function $\varphi:M\to \mathbb{R}$ satisfying $g' = \varphi g$.
Let $\conf(M)$ be the set of all conformal equivalence classes of $\mathcal{C}^\infty$-Riemannian metrics on $M$.

Let $\Lambda^+$ and $\Lambda^-$ be the sub-bundles of $\Lambda^2 := \Lambda^2(T^*M)$ consisting of self-dual and 
anti-self-dual $2$-forms with respect to $g_0$.
For $[g]\in \conf(M)$ we define a sub-bundle
$\Lambda^-_g\subset \Lambda^2$ as the set of anti-self-dual $2$-forms with respect to $g$.
There is a $\mathcal{C}^\infty$-bundle map $\mu_g:\Lambda^-\to \Lambda^+$ such that 
$|(\mu_g)_x|<1$ $(x\in M)$ and that $\Lambda_g^-$ is equal to the graph 
$\{\omega + \mu_g(\omega)|\, \omega\in \Lambda^-\}$.
Here $|(\mu_g)_x|<1$ $(x\in M)$ means that $|\mu_g(\omega)|< |\omega|$ for all non-zero $\omega\in \Lambda^-$.
($|\cdot|$ is the norm defined by $g_0$.)
From the previous argument, we get the following result.
\begin{corollary}  \label{cor: description of conformal structures, manifold case}
The map 
\[ \conf(M)\to \{\mu: \Lambda^-\to \Lambda^+: \text{$\mathcal{C}^\infty$-bundle map}|\, 
   |\mu_x|< 1 \> (x\in M)\}, \quad  [g]\mapsto \mu_g ,  \]
is bijective.
\end{corollary}

\subsection{Eigenvalues of the Laplacians on differential forms over $S^3$} 
\label{subsection: Eigenvalues of Laplacians on differential forms over S^3}
We will sometimes need estimates relating to lower bounds on the eigenvalues of the Laplacians on $S^3$.
Here the $3$-sphere $S^3$ is endowed with the Riemannian metric induced by the inclusion
$S^3 = \{x\in \mathbb{R}^4|\, x_1^2+x_2^2+x_3^2+x_4^2=1\} \subset \mathbb{R}^4$.
($\mathbb{R}^4$ has the standard Euclidean metric.)
The formal adjoint of $d:\Omega^i\to \Omega^{i+1}$ is denoted by $d^*:\Omega^{i+1}\to \Omega^i$.
\begin{lemma} \label{lemma: first non-zero eigenvalue of Laplacians on functions}
The first non-zero eigenvalue of the Laplacian $\Delta = d^*d$ acting on functions over $S^3$ is $3$.
\end{lemma}
\begin{proof}
See Sakai \cite[p. 272, Proposition 3.13]{Sakai}.
\end{proof}
\begin{lemma}
Let $\mathrm{Ker}(d^*)\subset \Omega^1$ be the space of $1$-forms $a$ over $S^3$ satisfying 
$d^* a=0$. 
Then the first eigenvalue of the Laplacian $\Delta = d^*d + d d^*$ acting on $\mathrm{Ker}(d^*)$ is 
$4$.
\end{lemma}
\begin{proof}
See Donaldson-Kronheimer \cite[p. 310, Lemma (7.3.4)]{Donaldson-Kronheimer}.
\end{proof}
As a corollary we get the following. (This is given in \cite[p. 310, Lemma (7.3.4)]{Donaldson-Kronheimer}.)
\begin{corollary} \label{cor: estimate relating to the first eigenvelue of Laplacian on 1-forms}
(i)
Let $a$ be a smooth $1$-form over $S^3$ satisfying $d^*a=0$. Then 
\[ \int_{S^3}|a|^2 d\vol \leq \frac{1}{4}\int_{S^3}|da|^2d\vol.\]

\noindent
(ii) For any smooth $1$-form $a$ on $S^3$, we have
\[   \left|\int_{S^3}a\wedge da\right| \leq \frac{1}{2}\int_{S^3}|da|^2 d\vol.\]
\end{corollary}
\begin{proof}
(i)
$\int |da|^2 = \int \langle a, \Delta a\rangle  \geq 4\int |a|^2$.

\noindent
(ii)
There is a smooth function $f$ on $S^3$ such that $b := a-df$ satisfies $d^*b=0$.
Then  
$\left|\int a\wedge da\right| = \left|\int b\wedge db\right| \leq \sqrt{\int |b|^2}\sqrt{\int |db|^2} 
   \leq \frac{1}{2}\int|db|^2 = \frac{1}{2}\int |da|^2$.
\end{proof}

\section{Decay estimate of instantons} \label{section: decay estimate of instatons}
\subsection{Classification of adapted connections} \label{subsection: classification of adapted connections}
Let us go back to the situation of Section \ref{subsection: construction}.
$Y$ is a simply connected, compact oriented $4$-manifold, and 
$X = Y^{\sharp \mathbb{Z}}$ is the connected sum of the infinite copies of $Y$ indexed by $\mathbb{Z}$. 
Since $X$ is non-compact, every principal $SU(2)$-bundle on it is 
isomorphic to the product bundle $E=X\times SU(2)$.
Following Donaldson \cite[Definition 3.5]{Donaldson}, we make the following definition.
\begin{definition} \label{def: adapted connection}
An adapted connection $A$ on $E$ is a connection on $E$ which is flat outside a compact set.
(That is, there is a compact set $L \subset X$ such that $F_A=0$ over $X\setminus L$.)
Two adapted connections $A_1$ and $A_2$ on $E$ are said to be equivalent as adapted connections if 
there is a gauge transformation $u:E\to E$ such that $u(A_1)$ is equal to $A_2$ outside a compact set.
\end{definition}
For $m\in \mathbb{Z}$, let $u_m:X\to SU(2)$ be a smooth map such that 
$(\rho_m)_* :H_3(X)\to H_3(SU(2))$ satisfies 
$(\rho_m)_*([S^3]) = m [SU(2)]$.
(Here $[S^3]$ is the fundamental class of the cross-section $S^3\subset  Y^{(n)}_T\cap Y^{(n)}_+$, 
and it is a generator of $H_3(X) \cong \mathbb{Z}$.
See Remark \ref{remark: orientation convention} below.)
This means that the restriction of $u_m$ to the cross-section $S^3\subset  Y^{(n)}_T\cap Y^{(n)}_+$
becomes a map of degree $m$ from $S^3$ to $SU(2)$ (for every $n\in \mathbb{Z}$).
\begin{remark} \label{remark: orientation convention}
The cross-section $S^3\subset  Y^{(n)}_T\cap Y^{(n)}_+$ is endowed with the orientation so that 
the identification $Y^{(n)}_T\cap Y^{(n)}_+ = (1,T-1)\times S^3$ is orientation preserving.
(The interval $(1,T-1)$ has the standard orientation.)
The orientation on the Lie group $SU(2)$ is chosen as follows:
Let $\theta\in \Omega^1\otimes su(2)$ be the left invariant $1$-form (on $SU(2)$) valued in 
the Lie algebra $su(2)$ satisfying 
$\theta(X)=X$ for all $X\in su(2) =T_1SU(2)$.
(In the standard notation, we can write $\theta = g^{-1}dg$ for $g\in SU(2)$.)
We choose the orientation on $SU(2)$ so that 
\begin{equation} \label{eq: orientation convention on SU(2)}
 \frac{1}{8\pi^2}\int_{SU(2)} tr\left(\theta\wedge d\theta + \frac{2}{3}\theta^3\right) = 
 \frac{-1}{24\pi^2}\int_{SU(2)} tr (\theta^3)= 1. 
\end{equation}
\end{remark}
Since $E$ is the product bundle, $u_m$ becomes a gauge transformation of $E$.
Let $\rho$ be the product flat connection on $E=X\times SU(2)$, and
set $\rho_m :=  u_m^{-1}(\rho)$.
Let $A(m)$ be a connection on $E$ which is equal to $\rho$ over $q^{-1}(-\infty,-1)$ and equal to 
$\rho_m$ over $q^{-1}(1,+\infty)$.
$A(m)$ is an adapted connection on $E$.
For $t>1$ we have 
\begin{equation} \label{eq: second Chern number of adapted connection A(m)}
 \frac{1}{8\pi^2}\int_X tr(F(A(m))^2) 
  = \frac{1}{8\pi^2}\int_{q^{-1}(t)} u_m^{*}\left(tr(\theta\wedge d\theta + \frac{2}{3}\theta^3)\right) = m.
\end{equation}
Here we have used (\ref{eq: orientation convention on SU(2)}) and 
$\deg(u_m|_{q^{-1}(t)}:q^{-1}(t) \to SU(2)) =m$.
\begin{proposition} \label{prop: classification of adapted connections}
For $m_1\neq m_2$, $A(m_1)$ and $A(m_2)$ are not equivalent as adapted connections.
If $A$ is an adapted connection on $E$, then $A$ is equivalent to $A(m)$ as an adapted connection where
\[ m = \frac{1}{8\pi^2}\int_X tr F_A^2 .\]
(An important point for us is that there are only countably many equivalence classes of adapted connections.)
\end{proposition}
\begin{proof}
The first statement follows from the equation (\ref{eq: second Chern number of adapted connection A(m)}).

Let $A$ be an adapted connection on $E$.
There is $M>0$ such that $A$ is flat on $q^{-1}(-\infty,-M]$ and $q^{-1}[M,\infty)$.
We choose $M>1$ so that $q^{-1}(M) = S^3\subset Y^{(n)}_T\cap Y^{(n)}_+$ and 
$q^{-1}(-M)=S^3\subset Y^{(-n)}_T\cap Y^{(-n)}_-$ for some $n>0$.
Since $q^{-1}(-\infty,-M]$ and $q^{-1}[M,\infty)$ are simply connected,
there are gauge transformations $u$ on $q^{-1}(-\infty,-M]$ and $u'$ on $q^{-1}[M,\infty)$ such that 
$u(A)=\rho$ and $u'(A)=\rho$.
We can extend $u$ all over $X$. Hence we can suppose that $u=1$ and that $A$ is equal to $\rho$ over $q^{-1}(-\infty,-M]$.
Set $m := \mathrm{deg}(u'|_{q^{-1}(M)}:q^{-1}(M)\to SU(2))$.
The degree of the map $(u_m^{-1}u')|_{q^{-1}(M)} :q^{-1}(M)\to SU(2)$ is zero.
Then there is a gauge transformation $u''$ of $E$ such that $u'' = u_m^{-1}u'$ on $q^{-1}[M,+\infty)$ 
and $u'' = 1$ on $q^{-1}(-\infty, M-1)$.
Then $u''(A)$ is equal to $\rho$ over $q^{-1}(-\infty,-M]$ and equal to $u_m^{-1}(\rho)$ over $q^{-1}[M,+\infty)$.
Hence $u''(A)$ is equal to $A(m)$ outside a compact set.
We have
\[ m= \frac{1}{8\pi^2}\int_X tr F(A(m))^2 = \frac{1}{8\pi^2}\int_X tr F_A^2 .\]
\end{proof}

\subsection{Preliminaries for the decay estimate}
We need the following. (This is a special case of \cite[Proposition 3.1, Remark 3.2]{Fukaya}.)
\begin{proposition} \label{prop: almost flat ASD connections}
Let $Z$ be a simply-connected compact Riemannian $4$-manifold with (or without) boundary, 
and $W \subset Z$ be a compact subset with $W \cap \partial Z=\emptyset$.
Then there are positive numbers $\varepsilon_1(W, Z)$ and $C_{1,k}(W, Z)$ $(k\geq 0)$ satisfying the following:
Let $A$ be an ASD connection on the product principal $SU(2)$-bundle over $Z$ satisfying 
$\norm{F_A}_{L^2(Z)}\leq \varepsilon_1(W, Z)$.
Then $A$ can be represented by a connection matrix $\tilde{A}$ over a neighborhood of $W$ satisfying 
\[ \norm{\tilde{A}}_{\mathcal{C}^k(W)} \leq C_{1,k}\norm{F_A}_{L^2(Z)}, \]
for all $k\geq 0$.
\end{proposition}
\begin{proof}
See Fukaya \cite[Proposition 3.1, Remark 3.2]{Fukaya}.
\end{proof}
\begin{lemma}\label{lemma: gauge over the tube}
Let $L>2$.
There exist positive numbers $\varepsilon_2$ and $C_{2,k}$ $(k\geq 0)$ independent of $L$ satisfying the following.
If $A$ is an ASD connection on the product principal $SU(2)$-bundle $G$ over $(0,L)\times S^3$ satisfying 
$\norm{F_A}_{L^2((0,L) \times S^3)} \leq \varepsilon_2$,
then $A$ can be represented by a connection matrix $\tilde{A}$ over a neighborhood of $[1,L-1]\times S^3$
satisfying 
\begin{equation} \label{eq: gauge over the tube}
 |\nabla^k \tilde{A}(t,\theta)| \leq C_{2,k}\norm{F_A}_{L^2((t-1,t+1)\times S^3)}, 
\end{equation}
for $(t, \theta) \in [1,L-1]\times S^3$ and $k\geq 0$.
\end{lemma}
\begin{proof}
Proposition \ref{prop: almost flat ASD connections} implies the following.
There exist positive numbers $\varepsilon'_2$ and $C'_{2,k}$ $(k\geq 0)$ such that 
if $B$ is an ASD connection on the product principal $SU(2)$-bundle over $[0,1]\times S^3$ 
satisfying $\norm{F_B}_{L^2([0,1]\times S^3)}\leq \varepsilon'_2$ then 
$B$ can be represented by a connection matrix $\tilde{B}$ over a neighborhood of $[1/4, 3/4]\times S^3$
satisfying 
\[ |\nabla^k \tilde{B}(x)| \leq C'_{2,k} \norm{F_B}_{L^2([0,1]\times S^3)} 
    \quad (x\in [1/4, 3/4]\times S^3,\, k\geq 0).\]
Let $\varepsilon_2$ be a small positive number with $\varepsilon_2<\varepsilon_2'$.
We will fix $\varepsilon_2$ later.
Suppose that $A$ is an ASD connection on the product principal $SU(2)$-bundle $G$ over $(0,L)\times S^3$ satisfying 
$\norm{F_A}_{L^2((0,L)\times S^3)}\leq \varepsilon_2$.
For $2 \leq n\leq [4L-4]$, set $I_n:=[n/4, n/4+1/2]$ and $J_n:=[n/4-1/4,n/4+3/4]$.
We have $I_n\subset J_n$.
For each $n$, there is a local trivialization $h_n$ of $G$ over a neighborhood of $I_n\times S^3$ such that 
the connection matrix $A_n := h_n(A)$ satisfies
\[ |\nabla^k A_n(x)|\leq C_{2,k}'\norm{F_A}_{L^2(J_n\times S^3)} \quad (x\in I_n\times S^3,\, k\geq 0).\]
Set $g_n:=h_{n+1}h_n^{-1}:(I_n\cap I_{n+1})\times S^3\to SU(2)$. 
Then 
$g_n(A_n) = A_{n+1}$ (i.e. $dg_n = g_n A_n -A_{n+1}g_n$).
In particular $|dg_n|\leq 4C_{2,0}' \varepsilon_2$. 
Fix a reference point $x_0\in S^3$.
By multiplying some constant gauge transformations on $h_n$'s, we can assume that 
$g_n(n/4+1/4,x_0)=1$.
Then
$|g_n-1|\leq \const \cdot \varepsilon_2$ over $(I_n\cap I_{n+1})\times S^3$ where $\const$ is independent of $L,n$.
Since the exponential map $\exp:su(2)\to SU(2)$ is locally diffeomorphic around $0\in su(2)$, 
if $\varepsilon_2$ is sufficiently small (but independent of $L, n$), we have 
$u_n := (\exp)^{-1}g_n: (I_n\cap I_{n+1})\times S^3\to su(2)$.
(Here we have fixed $\varepsilon_2>0$.)
Then
$g_n = e^{u_n}$ over $(I_n\cap I_{n+1})\times S^3$ with 
\[ |\nabla^k u_n(x)|\leq C_{2,k}''\norm{F_A}_{L^2((J_n\cup J_{n+1})\times S^3)} \quad 
   (x\in (I_n\cap I_{n+1})\times S^3, \, k\geq 0).\]
Let $\varphi$ be a smooth function in $\mathbb{R}$ such that $\supp (d\varphi)\subset (1/4, 1/2)$,
$\varphi(t) =0$ for $t\leq 1/4$ and $\varphi=1$ for $t\geq 1/2$.   
Set $\varphi_n(t) := \varphi(t-n/4)$. ($\supp(d\varphi_n)\subset \mathrm{Interior}(I_n\cap I_{n+1})$.)
We define a trivialization $h$ of $G$ over the union of $(I_n\cap I_{n+1})\times S^3$ ($2\leq n\leq [4L-4]$)
by setting $h:= e^{\varphi_n u_n}\circ h_n$ on $(I_n\cap I_{n+1})\times S^3$.
Then $h$ is smoothly defined over a neighborhood of $[1,L-1]\times S^3$, and the connection matrix 
$\tilde{A} := h(A)$ satisfies (\ref{eq: gauge over the tube}).
\end{proof}
Let us go back to the given manifolds $Y$ and $Y_T = p^{-1}(-T+1, T-1)$.
\begin{lemma}\label{lemma: gauge over Y_T}
Let $T> 4$.
There exist positive numbers $\varepsilon_3$ and $C_{3,k}$ $(k\geq 0)$ independent of $T$ satisfying the following.
If $A$ is an ASD connection on the product principal $SU(2)$-bundle over $Y_T$ satisfying $\norm{F_A}_{L^2(Y_T)}\leq \varepsilon_3$,
then $A$ can be represented by a connection matrix $\tilde{A}$ over $Y_{T-1}$ such that 
\[ |\nabla^k \tilde{A}(x)| \leq C_{3,k}\norm{F_A}_{L^2(p^{-1}(t-6,t+6)\cap Y_{T})} \quad (t=p(x)) ,\]
for $x\in Y_{T-1}$ and $k\geq 0$.
\end{lemma}
\begin{proof}
Set $Z:= p^{-1}[-3,3]$ and $W:=p^{-1}[-5/2,5/2] \subset Z$.
We apply Proposition \ref{prop: almost flat ASD connections} to these $Z$ and $W$:
There is $\varepsilon_3' >0$ (depending only on $Z$, $W$ and hence independent of $T$) such that 
if $\norm{F_A}_{L^2(Z)}\leq \varepsilon_3'$
then $A$ can be represented by a connection matrix $A_1$ over a neighborhood of $W$ such that 
\[ |\nabla^k A_1(x)| \leq \const_k \norm{F_A}_{L^2(Z)} \quad (x\in W,\, k\geq 0).\]
On the other hand, by applying Lemma \ref{lemma: gauge over the tube} to 
the tubes $p^{-1}(-T+1, -1) = (-T+1,-1) \times S^3$ and $p^{-1}(1,T-1) = (1,T-1) \times S^3$, 
if $\norm{F_A}_{L^2(Y_T)}\leq \varepsilon_2$ (the positive constant introduced in Lemma \ref{lemma: gauge over the tube})
 then $A$ can be represented by a connection matrix $A_2$ over 
a neighborhood of $p^{-1}[-T+2,-2] \sqcup p^{-1}[2,T-2] = [-T+2,-2]\times S^3 \sqcup [2,T-2]\times S^3$ 
such that 
\[ |\nabla^k A_2(t,\theta)| \leq C_{2,k} \norm{F_A}_{L^2((t-1,t+1)\times S^3)} \quad 
((t, \theta)\in [-T+2,-2]\times S^3 \sqcup [2,T-2]\times S^3, \, k\geq 0).\]
Then by patching $A_1$ and $A_2$ over $p^{-1}(-5/2, -2)$ and $p^{-1}(2,5/2)$ as in the proof of 
Lemma \ref{lemma: gauge over the tube}, we get the desired connection matrix $\tilde{A}$.
\end{proof}

\subsection{Exponential decay} \label{subsection: exponential decay}
In this subsection we study a decay estimate of instantons on 
the product principal $SU(2)$-bundle $E = X\times SU(2)$.
The results in this section will be used in Section \ref{section: proof of the main theorem}.
Our method is based on the arguments of Donaldson \cite[Section 4.1]{Donaldson} and 
Donaldson-Kronheimer \cite[Section 7.3]{Donaldson-Kronheimer}.
In this subsection we always suppose $T>4$.
Let $g$ be a Riemannian metric on $X$ which is equal to $g_0$ (the Riemannian metric given 
in Section \ref{subsection: construction}) outside a compact set.
Let $A$ be a $g$-ASD connection on $E$ satisfying 
\[ \int_X |F_A|_g^2 d\vol_g < \infty.\]
For $t\in \mathbb{R}$, set 
\[ J(t) := \int_{q^{-1}(t,+\infty)} |F_A|_g^2d\vol_g.\]
For $t\gg 1$ we have $J(t) = \int_{q^{-1}(t,+\infty)} |F_A|^2 d\vol$ where 
$|\cdot|$ and $d\vol$ are the norm and volume form with respect to the periodic metric $g_0$.
Recall that for each integer $n$ we have $q^{-1}(nT+1, (n+1)T-1) = Y_T^{(n)}\cap Y^{(n)}_+ = (1,T-1)\times S^3$.
\begin{lemma} \label{lemma: differential inequality for J}
There is $n_0(A)>0$ such that for $n\geq n_0(A)$
\[ J'(t) \leq -2J(t) \quad (nT+2\leq t\leq (n+1)T-2).\]
(The value $-2$ is not optimal.)
\end{lemma}
\begin{proof}
In this proof we always suppose $nT +2\leq t\leq (n+1)T-2$ and $n\gg 1$. We have 
\[ J'(t) = -\int_{q^{-1}(t)}|F_A|^2d\vol = -2\norm{F(A_t)}_{L^2(S^3)}^2 \quad (A_t := A|_{q^{-1}(t)}).\]
Here we have used the fact $|F_A|^2 = 2|F(A_t)|^2$. This is the consequence of the ASD condition.
From Lemma \ref{lemma: gauge over the tube}, 
we can assume that, for $n\gg 1$,  a connection matrix of $A$ over $q^{-1}[nT+2, (n+1)T-2]$ is as small as we want 
with respect to the $\mathcal{C}^1$-norm (or any other $\mathcal{C}^k$-norm).
In particular we have $\norm{F(A_t)}_{L^2} \ll 1$ for $n\gg 1$.
Then, by using \cite[Proposition 4.4.11]{Donaldson-Kronheimer}, we can suppose that 
$A_t$ is represented by a connection matrix satisfying 
\begin{equation} \label{eq: bound on L^2_1 norm}
 \norm{A_t}_{L^2_1(S^3)} \leq \const \norm{F(A_t)}_{L^2(S^3)}.
\end{equation}
Then we can prove 
\begin{sublemma}
 \[ J(t) = -\int_{S^3} \mathrm{tr}(A_t\wedge dA_t+\frac{2}{3}A_t^3)\quad (=: -\theta(A_t)).\]
\end{sublemma}
\begin{proof}
For $m>n\gg 1$ and $mT+2\leq s\leq (m+1)T-2$, 
\begin{equation} \label{eq: Chern-Simon-Storkes mod Z}
 \int_{q^{-1}[t,s]} |F_A|^2 d\vol \equiv \theta(A_s)-\theta(A_t) \mod 8\pi^2 \mathbb{Z}.
\end{equation}
We can suppose that the connection matrix $A_s$ also satisfies (\ref{eq: bound on L^2_1 norm}).
Then both of the left and right hand sides of the above equation (\ref{eq: Chern-Simon-Storkes mod Z}) are sufficiently small.
Hence
\[ \int_{q^{-1}[t,s]} |F_A|^2 d\vol = \theta(A_s)-\theta(A_t).\]
We have $\theta(A_s)\to 0$ as $m\to +\infty$. Then we get the above result.
\end{proof}
From Corollary \ref{cor: estimate relating to the first eigenvelue of Laplacian on 1-forms} (ii), 
\[ \left|\int_{S^3}\mathrm{tr}(A_t\wedge dA_t)\right| \leq \frac{1}{2}\int_{S^3}|dA_t|^2.\]
Since $dA_t = F(A_t)-A_t^2$, 
$\norm{dA_t}_{L^2(S^3)}^2\leq \norm{F(A_t)}_{L^2}^2 + 2\norm{F(A_t)}_{L^2}\norm{A_t^2}_{L^2} + \norm{A_t^2}_{L^2}^2$.
We have $L^2_1(S^3)\hookrightarrow L^6(S^3)$. Hence 
$\norm{A_t^2}_{L^2}\leq \const \norm{A_t}_{L^2_1}^2\leq \const \norm{F(A_t)}_{L^2}^2$ by (\ref{eq: bound on L^2_1 norm}).
Hence $\norm{dA_t}_{L^2}^2 \leq (1+ \const \norm{F(A_t)}_{L^2})\norm{F(A_t)}_{L^2}^2$.
In a similar way, we have 
\[ \left|\int_{S^3}\mathrm{tr}(A_t^3) \right|\leq \const \norm{A_t}_{L^3}^3\leq \const \norm{F(A_t)}_{L^2}^3.\]
Thus we have 
\[ J(t) = -\theta(A_t) \leq \left(\frac{1}{2}+\const\norm{F(A_t)}_{L^2}\right)\norm{F(A_t)}_{L^2}^2.\]
Since $J'(t)= -2\norm{F(A_t)}_{L^2}^2$ and $\norm{F(A_t)}_{L^2}\ll 1$, we have
$J(t) \leq \norm{F(A_t)}_{L^2}^2 = -\frac{1}{2}J'(t)$.
Hence $J'(t)\leq -2J$.
\end{proof}
\begin{corollary} \label{cor: decay estimate of J}
For $t\geq n_0(A)T+2$,
\[ J(t) \leq \const_{A, T} \cdot e^{-2(1-4/T)t}.\]
Here $\const_{A,T}$ is a positive constant depending on $A$ and $T$.
\end{corollary}
\begin{proof}
First note that $J(t)$ is monotone non-increasing. 
For $nT+2\leq t\leq (n+1)T-2$ $(n\geq n_0(A) =:n_0)$, we have $J(t)\leq e^{-2(t-nT-2)}J(nT+2)$ 
by Lemma \ref{lemma: differential inequality for J}.

Set $a_n := J(nT+2)$ $(n\geq n_0)$.
$a_{n+1}\leq J((n+1)T-2) \leq e^{-2(T-4)}a_n$.
Hence $a_n\leq e^{-2(T-4)(n-n_0)}a_{n_0}$.

For $nT+2\leq t\leq (n+1)T-2$, 
$J(t) \leq e^{-2(t-nT-2)}a_n\leq e^{-2(t-4n)}e^{4+2n_0T-8n_0}a_{n_0}$.
Since $t\geq nT+2$, we have $t-4n \geq (1-4/T)t+8/T$. Hence 
$J(t)\leq \const_{A, T}e^{-2(1-4/T)t}$.

For $(n+1)T-2< t< (n+1)T+2$, 
$J(t)\leq J((n+1)T-2) \leq \const'_{A, T}e^{-2(1-4/T)t}$.
\end{proof}
In the same way we can prove the following.
\begin{lemma} \label{lemma: decay estimate of energy for minus direction}
For $t\gg 1$ we have 
\[ \int_{q^{-1}(-\infty, -t)} |F_A|^2 d\vol \leq \const_{A, T} \cdot e^{-2(1-4/T)t}.\]
\end{lemma}
\begin{corollary} \label{cor: existence of adapted connection to which instanton converges exponentially} 
There exists an adapted connection $A_0$ on $E$ satisfying 
\[ |\nabla_{A_0}^k(A(x)-A_0(x))| \leq \const_{k,A,T}\cdot e^{-(1-4/T)|t|} \quad (t=q(x)), \]
for all integers $k\geq 0$.
\end{corollary}
\begin{proof}
For $|n| \gg 1$, we have $\norm{F_A}_{L^2(Y^{(n)}_T)} \leq \varepsilon_3$. 
($\varepsilon_3$ is a positive constant introduced in Lemma \ref{lemma: gauge over Y_T}.)
Then by Lemma \ref{lemma: gauge over Y_T}, Corollary \ref{cor: decay estimate of J} and 
Lemma \ref{lemma: decay estimate of energy for minus direction}, 
$A$ can be represented by a connection matrix $A_n$
on $Y^{(n)}_{T-1}$ $(|n|\gg 1)$ such that 
\[ |\nabla^k A_n (x)|\leq \const_{k,A,T} \cdot e^{-(1-4/T)|t|} \quad (x\in Y^{(n)}_{T-1},\, t=q(x), \, k\geq 0).\]
By patching these connection matrices over $Y^{(n)}_{T-1} \cap Y^{(n+1)}_{T-1}$ $(|n|\gg 1)$
as in the proof of Lemma \ref{lemma: gauge over the tube}, 
$A$ can be represented by a connection matrix $\tilde{A}$ on $\{|t|\gg 1\}$ such that 
\[ |\nabla^k \tilde{A}(x)| \leq \const_{k,A,T}\cdot e^{-(1-4/T)|t|}\quad (|t|\gg 1, \, k\geq 0).\]
To be more precise, there are $t_0 \gg 1$ and a trivialization 
$h:E|_{\{|t|> t_0\}} \to \{|t|>t_0\}\times SU(2)$ such that 
$h(A)$ satisfies
\[ |\nabla_\rho^k (h(A)-\rho)| \leq \const_{k, A, T}\cdot e^{-(1-4/T)|t|} \quad (|t| > t_0,\, k\geq 0),\]
where $\rho$ is the product connection. 
This means that 
\[ |\nabla_{h^{-1}(\rho)}^k (A-h^{-1}(\rho)) |\leq \const_{k,A,T}\cdot e^{-(1-4/T)|t|} \quad (|t| > t_0,\, k\geq 0).\]
Take a connection $A_0$ on $E$ which is equal to $h^{-1}(\rho)$ over $\{|t|\geq t_0 +1\}$.
Then $A_0$ is an adapted connection satisfying the desired property.
\end{proof}

\section{Preliminaries for linear theory} \label{section: preliminaries for linear theory}
In this section, we study differential operators over $X$.
The results in this section will be used in Section \ref{section: linear theory}.
All arguments in Sections \ref{subsection: preliminary estmates over the tube} 
and \ref{subsection: preliminary results over Y^hat} are essentially given in 
Donaldson \cite[Chapters 3 and 4]{Donaldson}.
\subsection{Preliminary estimates over the tube} \label{subsection: preliminary estmates over the tube}
Let $\alpha$ be a real number with $0<|\alpha|<1$. 
In this subsection we study some differential operators over $\mathbb{R}\times S^3$.
We denote $t$ as the parameter of the $\mathbb{R}$-factor 
(i.e. the natural projection $t:\mathbb{R}\times S^3\to \mathbb{R}$).
Let $d^*:\Omega^1_{\mathbb{R}\times S^3}\to \Omega^0_{\mathbb{R}\times S^3}$ 
be the formal adjoint of the derivative $d:\Omega^0_{\mathbb{R}\times S^3}\to \Omega^1_{\mathbb{R}\times S^3}$ 
over $\mathbb{R}\times S^3$.
We have $d^* = -*d*$ where $*$ is the Hodge star over $\mathbb{R}\times S^3$.
We define a differential operator $d^{*,\alpha}:\Omega^1_{\mathbb{R}\times S^3}\to \Omega^0_{\mathbb{R}\times S^3}$ by setting 
$d^{*,\alpha} b: = e^{-2\alpha t}d^{*}(e^{2\alpha t}b)$ $(b\in \Omega^1_{\mathbb{R}\times S^3})$.
Then
\[ d^{*,\alpha}b = d^*b -2\alpha *(dt\wedge *b).\]
If $f\in \Omega^0_{\mathbb{R}\times S^3}$ and $b\in \Omega^1_{\mathbb{R}\times S^3}$ have 
compact supports, then 
\[ \int_{\mathbb{R}\times S^3}e^{2\alpha t}\langle df, b\rangle d\vol 
= \int_{\mathbb{R}\times S^3} e^{2\alpha t}\langle f, d^{*,\alpha}b\rangle d\vol.\]
Consider $d^+ := \frac{1}{2}(1+*)d :\Omega^1_{\mathbb{R}\times S^3}\to \Omega^+_{\mathbb{R}\times S^3}$, and set 
$D^\alpha := -d^{*,\alpha}+d^+ :\Omega^1_{\mathbb{R}\times S^3} \to 
\Omega^0_{\mathbb{R}\times S^3}\oplus \Omega^+_{\mathbb{R}\times S^3}$.

Let $\Lambda^i_{S^3}$ $(i\geq 0)$ be the bundle of $i$-forms over $S^3$.
Consider the pull-back of $\Lambda^i_{S^3}$ by the projection $\mathbb{R}\times S^3\to S^3$, 
and we also denote it as $\Lambda^i_{S^3}$ for simplicity.
We can identify the bundle $\Lambda^1_{\mathbb{R}\times S^3}$ of 
$1$-forms on $\mathbb{R}\times S^3$ with the bundle $\Lambda^0_{S^3}\oplus \Lambda^1_{S^3}$ by
\[ \Lambda^0_{S^3}\oplus \Lambda^1_{S^3} \ni (b_0, \beta) \longleftrightarrow 
     b_0dt + \beta \in \Lambda^1_{\mathbb{R}\times S^3}.\]
We also naturally identify the bundle $\Lambda^0_{\mathbb{R}\times S^3}$ with $\Lambda^0_{S^3}$.
The bundle $\Lambda^+_{\mathbb{R}\times S^3}$ of self-dual forms can be identified with the bundle $\Lambda^1_{S^3}$ by 
\[ \Lambda_{S^3}^1\ni \beta \longleftrightarrow 
  \frac{1}{2}(dt\wedge \beta + *_3\beta) \in \Lambda^+_{\mathbb{R}\times S^3} \quad 
    (\text{$*_3$: the Hodge star on $S^3$}).\]
We define $L:\Gamma (\Lambda^0_{S^3}\oplus \Lambda^1_{S^3}) \to \Gamma (\Lambda^0_{S^3}\oplus \Lambda^1_{S^3})$
by setting 
\[ L \begin{pmatrix} b_0 \\ \beta \end{pmatrix} := 
  \begin{pmatrix} 0 & -d_3^*\\ -d_3 & *_3d_3\end{pmatrix} 
  \begin{pmatrix} b_0 \\ \beta \end{pmatrix} ,\]
where $d_3$ is the exterior derivative on $S^3$ and $d_3^* = -*_3d_3*_3$.
Let $b = (b_0, \beta)\in \Gamma(\Lambda^0_{S^3}\oplus\Lambda^1_{S^3}) =\Omega^1_{\mathbb{R}\times S^3}$
(i.e. $b=b_0 dt + \beta$).
Then $D^\alpha b\in \Omega^0_{\mathbb{R}\times S^3}\oplus \Omega^1_{\mathbb{R}\times S^3} 
= \Gamma(\Lambda^0_{S^3}\oplus \Lambda^1_{S^3})$ is given by 
\[ D^\alpha b = \frac{\partial}{\partial t}\begin{pmatrix} b_0 \\ \beta \end{pmatrix}
   + \left(L + \begin{pmatrix} 2\alpha & 0 \\ 0 & 0 \end{pmatrix}\right)
    \begin{pmatrix} b_0\\ \beta \end{pmatrix} .\]
For $u\in \Omega^i_{\mathbb{R}\times S^3}$ $(i\geq 0)$, 
we define the Sobolev norm $\norm{u}_{L^{2}_k}$ $(k\geq 0)$ by 
\begin{equation} \label{eq: sobolev norm}
 \norm{u}^2_{L^{2}_k} := \sum_{j=0}^k \int_{\mathbb{R}\times S^3} |\nabla^j u|^2 d\vol .
\end{equation}
We define the weighted Sobolev norm $\norm{u}_{L^{2,\alpha}_k}$ by 
\begin{equation} \label{eq: weighted Sobolev norm}
 \norm{u}_{L^{2,\alpha}_k} := \norm{e^{\alpha t}u}_{L^2_k}.
\end{equation}
The map 
$L^{2,\alpha}_k(\mathbb{R}\times S^3, \Lambda^i_{\mathbb{R}\times S^3}) 
\ni u\mapsto e^{\alpha t}u\in L^2_k(\mathbb{R}\times S^3,\Lambda^i_{\mathbb{R}\times S^3})$ is an isometry.

$D^\alpha$ becomes a bounded linear map from 
$L^{2,\alpha}_{k+1}(\mathbb{R}\times S^3, \Lambda^1_{\mathbb{R}\times S^3})$
to $L^{2,\alpha}_k(\mathbb{R}\times S^3, \Lambda^0_{\mathbb{R}\times S^3}\oplus \Lambda^+_{\mathbb{R}\times S^3})$.
For $b = b_0 dt + \beta \in \Omega^1_{\mathbb{R}\times S^3}$ as above, we have 
\begin{equation} \label{eq: variant of D^alpha} 
 e^{\alpha t} D^{\alpha}(e^{-\alpha t}b) = \frac{\partial}{\partial t} \begin{pmatrix} b_0 \\ \beta \end{pmatrix}
   +  \left(L + \begin{pmatrix} \alpha & 0 \\ 0 & -\alpha \end{pmatrix}\right)
    \begin{pmatrix} b_0\\ \beta \end{pmatrix} .
\end{equation}
Set
\[ L^\alpha := L + \begin{pmatrix} \alpha & 0 \\ 0 & -\alpha \end{pmatrix}.\]
Recall that we have assumed $0<|\alpha|<1$.
\begin{lemma} \label{lemma: eigenvalue of L^alpha}
Consider $L^\alpha$ as an essentially self-adjoint elliptic differential operator acting on 
$\Omega_{S^3}^0\oplus\Omega^1_{S^3}$ over $S^3$.
If $\lambda$ is an eigenvalue of $L^\alpha$, then $|\lambda|\geq |\alpha|$.
Moreover if $\lambda\neq \alpha$, then $|\lambda| >1$.
\end{lemma}
\begin{proof}
We have 
\[ \Omega^0_{S^3}\oplus \Omega^1_{S^3} = (\Omega^0_{S^3}\oplus d_3(\Omega^0_{S^3}))\oplus \ker d_3^*,\]
where $d_3^* = -*_3 d_3*_3:\Omega^1_{S^3}\to \Omega^0_{S^3}$.
The subspaces $\Omega^0_{S^3}\oplus d_3(\Omega^0_{S^3})$ and $\ker d_3^*$ are both $L^\alpha$-invariant.

For $\beta\in \ker d_3^*$, $L^\alpha (0, \beta) = (0, *_3d_3\beta -\alpha \beta)$.
Suppose that $L^\alpha (0,\beta) = \lambda (0, \beta)$ and $\beta$ is not zero. 
Since $d_3^*\beta=0$ and $H^1(S^3)=0$, we have $d_3\beta\neq 0$.
Then $*_3d_3\beta = (\lambda+\alpha)\beta$ and $\lambda+\alpha\neq 0$.
Since we have (Corollary \ref{cor: estimate relating to the first eigenvelue of Laplacian on 1-forms} (ii))
\[ \left|\int_{S^3} \beta\wedge d_3\beta\right| \leq \frac{1}{2}\int_{S^3}|d_3\beta|^2 d\vol \]
and $(\lambda + \alpha) \beta\wedge d_3\beta =|d_3\beta|^2 d\vol$, we have 
\[ 2\leq |\lambda + \alpha|.\]
Then $|\lambda|\geq 2-|\alpha| >1>|\alpha|$.

For $(f, d_3 g)\in \Omega^0_{S^3}\oplus d_3(\Omega^0_{S^3})$ ($f$ and $g$ are smooth functions on $S^3$), 
\[ L^\alpha \begin{pmatrix} f \\ d_3g\end{pmatrix} = \begin{pmatrix} \alpha f - \Delta_3 g \\ -d_3f -\alpha d_3g \end{pmatrix},
  \quad (\text{$\Delta_3 = d_3^*d_3$ is the Laplacian on functions over $S^3$}).\]
Suppose that $L^\alpha (f, d_3g) = \lambda (f, d_3g)$ and $(f,d_3g)$ is not zero. 
Then 
\[ \Delta_3 g = (\alpha-\lambda) f, \quad d_3f= -(\alpha+\lambda)d_3g.\]
\textbf{Case 1}: Suppose $\alpha +\lambda=0$. Then $f$ is a constant, and 
\[  0 = \int_{S^3} \Delta_3 g \, d\vol = 2\alpha \int_{S^3} fd\vol .\]
Hence $f\equiv 0$. This implies $\Delta_3 g\equiv 0$ and hence $d_3 g\equiv 0$.
This is a contradiction.

\noindent
\textbf{Case 2}: Suppose $\alpha+\lambda\neq 0$.
Then $\Delta_3 f = (\lambda^2-\alpha^2)f$.
Since the first non-zero eigenvalue of the Laplacian $\Delta_3$ is $3$ 
(Lemma \ref{lemma: first non-zero eigenvalue of Laplacians on functions}),
$\lambda^2-\alpha^2=0$ or $\lambda^2-\alpha^2\geq 3$.
Since $\lambda\neq -\alpha$, we have 
\[ \lambda = \alpha \quad \text{or} \quad |\lambda|\geq \sqrt{3 + \alpha^2} \geq \sqrt{3}.\]
\end{proof}
\begin{lemma}\label{lemma: L^2,alpha estimate}
For $a\in L^{2,\alpha}_1(\mathbb{R}\times S^3, \Lambda^1_{\mathbb{R}\times S^3})$,
we have $\norm{a}_{L^{2,\alpha}}\leq \sqrt{2}|\alpha|^{-1}\norm{D^\alpha a}_{L^{2,\alpha}}$.
Moreover $\norm{a}_{L^{2,\alpha}_1}\leq \const_\alpha \norm{D^\alpha a}_{L^{2,\alpha}}$.
\end{lemma}
\begin{proof}
We can suppose that $a$ is smooth and compact supported.
Set $b := e^{\alpha t}a = b_0 dt + \beta$ where 
$(b_0,\beta)\in \Gamma (\mathbb{R}\times S^3, \Lambda^0_{S^3}\oplus \Lambda^1_{S^3})$.
Let $\{\varphi_\lambda\}_\lambda$ be a complete orthonormal basis of $L^2(S^3, \Lambda_{S^3}^0\oplus \Lambda_{S^3}^1)$
consisting of eigen-functions of $L^\alpha$ over $S^3$ with $L^\alpha \varphi_\lambda = \lambda \varphi_\lambda$
where $\lambda$ runs over all eigenvalues of $L^\alpha$.
From Lemma \ref{lemma: eigenvalue of L^alpha}, we have $|\lambda|\geq |\alpha|$.
Decompose $(b_0,\beta)$ by $\{\varphi_\lambda\}$ as 
\[ (b_0(t, \theta),\beta(t,\theta)) = \sum_\lambda c_\lambda(t) \varphi_\lambda(\theta).\]
Since $a$ is compact supported, the functions $c_\lambda$ are also compact supported.
$(\partial/\partial t + L^\alpha)(b_0,\beta) = \sum_{\lambda} (c_\lambda'(t) + \lambda c_\lambda(t))\varphi_\lambda$.
If $(\partial/\partial t + L^\alpha)(b_0,\beta) = (b_1,\gamma)$, then 
$e^{\alpha t}D^\alpha (e^{-\alpha t}b) = 
(b_1, \frac{1}{2}(dt\wedge \gamma + *_3 \gamma))$.
Hence $|e^{\alpha t}D^\alpha (e^{-\alpha t}b)| = \sqrt{|b_1|^2 + |\gamma|^2/2} \geq 
|(\partial/\partial t + L^\alpha)(b_0,\beta)|/\sqrt{2}$.
Therefore
\[ \int_{\mathbb{R}\times S^3}|e^{\alpha t}D^\alpha (e^{-\alpha t}b)|^2 d\vol 
   \geq  \frac{1}{2}\sum_\lambda \int_{-\infty}^\infty |c_\lambda' + \lambda c_\lambda|^2dt .\] 
$|c_\lambda' + \lambda c_\lambda|^2 = |c_\lambda'|^2 + \lambda (c_\lambda^2)' + \lambda^2 c_\lambda^2$.
Since $|\lambda|\geq |\alpha|$ and the functions $c_\lambda$ are compact supported,
\[ \sum_\lambda \int_{-\infty}^\infty |c_\lambda' + \lambda c_\lambda|^2dt\geq 
\alpha^2 \sum_\lambda \int_{-\infty}^\infty |c_\lambda|^2dt = \alpha^2 \norm{b}^2_{L^2}.\]
Then 
\[ \norm{D^\alpha a}_{L^{2,\alpha}} = \norm{e^{\alpha t}D^\alpha (e^{-\alpha t}b)}_{L^2} 
\geq \frac{|\alpha|}{\sqrt{2}}\norm{b}_{L^2}
 = \frac{|\alpha|}{\sqrt{2}}\norm{a}_{L^{2,\alpha}}.\]

Since $e^{\alpha t}D^\alpha e^{-\alpha t} = \frac{\partial}{\partial t} + L^\alpha$ is a translation invariant 
elliptic differential operator, for every $n\in \mathbb{Z}$ we have 
\[ \norm{b}^2_{L^2_1((n,n+1)\times S^3)} \leq 
   \const_\alpha \left(\norm{b}^2_{L^2((n-1,n+2)\times S^3)} 
 + \norm{e^{\alpha t}D^\alpha(e^{-\alpha t} b)}^2_{L^2((n-1,n+2)\times S^3)}\right).\]
Here $\const_\alpha$ is independent of $n$.
By summing up this estimate over $n\in \mathbb{Z}$, we get 
\[ \norm{b}_{L^2_1(\mathbb{R}\times S^3)} \leq 
   \const_\alpha \left(\norm{b}_{L^2(\mathbb{R}\times S^3)}
   +\norm{e^{\alpha t}D^\alpha(e^{-\alpha t}b)}_{L^2(\mathbb{R}\times S^3)}\right).\]
This shows $\norm{a}_{L_1^{2,\alpha}}\leq \const_\alpha (\norm{a}_{L^{2,\alpha}}+\norm{D^\alpha a}_{L^{2,\alpha}})
 \leq \const'_\alpha \norm{D^\alpha a}_{L^{2,\alpha}}$. 
\end{proof}
\begin{lemma} \label{lemma: pointwise estimate over the negative half tube}
(i) 
Suppose $\alpha>0$.
Let $a$ be a smooth $1$-form over the negative half tube $(-\infty, 0)\times S^3$ satisfying
$\int_{(-\infty,0)\times S^3}e^{2\alpha t} |a|^2 d\vol < +\infty$.
Suppose $D^\alpha a= 0$. Then 
\[ |a|, |\nabla a|\leq \const_{a,\alpha} e^{(1-\alpha)t} \quad (t<-2).\]

\noindent 
(ii) Suppose $\alpha<0$. Let $a$ be a smooth $1$-form over the positive half tube $(0,+\infty)\times S^3$ satisfying
$\int_{(0,+\infty)\times S^3}e^{2\alpha t}|a|^2 d\vol < +\infty$, and suppose $D^\alpha a =0$. Then 
\[ |a|, |\nabla a|\leq \const_{a,\alpha}e^{-(1+\alpha)t} \quad (t>2).\]
\end{lemma}
\begin{proof}
We give the proof of the case (i) ($\alpha >0$).
The case (ii) can be proved in the same way.
Set $b := e^{\alpha t}a = b_0 dt + \beta$ where 
$(b_0,\beta)\in \Gamma(\mathbb{R}\times S^3, \Lambda^0_{S^3}\oplus \Lambda^1_{S^3})$.
Then $e^{\alpha t} D^{\alpha}(e^{-\alpha t}b) =0$.
Choose $\{\varphi_\lambda\}_\lambda$ as in the proof of Lemma \ref{lemma: L^2,alpha estimate}.
Decompose $(b_0,\beta)$ by $\{\varphi_\lambda\}$ as 
$(b_0(t, \theta),\beta(t,\theta)) = \sum_\lambda c_\lambda(t) \varphi_\lambda(\theta)$.
Since $(\partial/\partial t + L^\alpha)(b_0,\beta) = \sum (c_\lambda'(t) + \lambda c_\lambda(t))\varphi_\lambda =0$,
we have $c_\lambda(t) = d_\lambda e^{-\lambda t}$ where $d_\lambda$ is a constant.
For $t<0$,
\[ \int_{\{t\}\times S^3} |b|^2d\vol_3 = \sum_\lambda |c_\lambda|^2 = \sum_{\lambda}|d_\lambda|^2 e^{-2\lambda t}
   \geq |d_\lambda|^2 e^{-2\lambda t} .\]
Since the $L^2$-norm of $b$ over $(-\infty, 0)\times S^3$ is finite, we have $d_\lambda=0$ for $\lambda \geq 0$.
Set 
\[ B := e^2\int_{\{-1\}\times S^3}|b|^2 d\vol_3 = e^2\sum_{\lambda<0} |d_\lambda e^\lambda|^2 <\infty.\]
From Lemma \ref{lemma: eigenvalue of L^alpha}, negative eigenvalues $\lambda$ satisfy $\lambda < -1$.
Hence for $t<-1$
\[ \int_{\{t\}\times S^3}|b|^2 d\vol_3 = \sum_{\lambda <0} |d_\lambda e^\lambda|^2 e^{-2\lambda(t+1)}
   \leq \sum_{\lambda<0} |d_\lambda e^\lambda|^2 e^{2(t+1)} = B e^{2t}.\]
Then for $t<-2$,
\[ \int_{(t-1,t+1)\times S^3}|b|^2 d\vol \leq B \int_{t-1}^{t+1}e^{2s}ds \leq B e^{2(t+1)} .\]
Since $e^{\alpha t} D^{\alpha}(e^{-\alpha t}b)=0$ (and this is a translation invariant equation), 
the elliptic regularity implies 
\[ |b|, |\nabla b|\leq \const_{\alpha}\sqrt{B}\cdot e^{t} \quad (t<-2).\]
(Indeed we can choose $\const_{\alpha}$ independent of $\alpha$. But it is unimportant for us.)
Since $a=e^{-\alpha t}b$, we have 
\[ |a|, |\nabla a|\leq \const'_{a,\alpha} e^{(1-\alpha)t} \quad (t<-2).\]
\end{proof}

\subsection{Preliminary results over $\hat{Y}$} \label{subsection: preliminary results over Y^hat}
Recall that $Y$ is a simply connected closed oriented $4$-manifold and that 
$\hat{Y}=Y\setminus \{x_1,x_2\}$.
$\hat{Y}$ has cylinderical ends, and
we have $p:\hat{Y}\to \mathbb{R}$.
For a section of $u$ of $\Lambda^i$ $(i\geq 0)$ over $\hat{Y}$, we define the Sobolev norm $\norm{u}_{L^2_k}$ $(k\geq 0)$
as in (\ref{eq: sobolev norm}). 
We define the weighted Sobolev norm by $\norm{u}_{L^{2,\alpha}_k} := \norm{e^{\alpha t} u}_{L^{2,\alpha}}$
where $t=p(x)$ $(x\in \hat{Y})$.
Recall $0<|\alpha|<1$.

For a $1$-form $a$ over $\hat{Y}$ we set 
$D^\alpha a := -d^{*,\alpha} a + d^+ a = -e^{-2\alpha t}d^*(e^{2\alpha t}a) + d^+a$.
\begin{lemma}\label{lemma: the first cohomology is zero}
Let $a$ be a $1$-form over $\hat{Y}$ with $\norm{a}_{L^{2,\alpha}_1} <\infty$.
If $D^\alpha a=0$, then $a=0$.
\end{lemma}
\begin{proof}
We give the proof of the case $\alpha>0$.
The case $\alpha<0$ can be proved in the same way.
We divide the proof into three steps.

\textbf{Step 1}: We will show that the above assumption implies $da=0$.
First we want to show $a, da \in L^2$.
We have 
\[ \int_{t>0}|a|^2 d\vol \leq \int_{t>0}e^{2\alpha t}|a|^2 d\vol <\infty, \quad
   \int_{t>0}|da|^2 d\vol \leq \int_{t>0}e^{2\alpha t}|da|^2 d\vol <\infty.\] 
Lemma \ref{lemma: pointwise estimate over the negative half tube} implies that
the $L^2$-norms of $a$ and $da$ over $Y_- = (-\infty, -1)\times S^3$ are finite.
Hence $a, da\in L^2$.
For $R>1$, let $\beta_R$ be a smooth function over $\hat{Y}$ such that 
$\beta_R=1$ over $p^{-1}(-R,R)$, $\beta_R =0$ over $p^{-1}(-\infty,-2R)\cup p^{-1}(2R,\infty)$
and $|d\beta_R|\leq 2/R$.
\[ 0= \int d(\beta_Ra\wedge da) = \int \beta_Rda\wedge da + \int d\beta_R \wedge a\wedge da.\]
Since $d^+a =0$, we have $da\wedge da = -|da|^2d\vol$ and hence 
\[ \int \beta_R|da|^2d\vol = \int d\beta_R\wedge a\wedge da \leq \frac{2}{R}\norm{a}_{L^2}\norm{da}_{L^2}.\]
Let $R\to +\infty$. Then $\int |da|^2 d\vol =0$.
Hence $da =0$.

\textbf{Step 2}:
We have
\[ |a|\leq \const_{a, \alpha}e^{-\alpha t}\> (t>1), \quad |a|\leq \const_{a,\alpha} e^{(1-\alpha)t} \>(t<-1).\]
The latter estimate comes from Lemma \ref{lemma: pointwise estimate over the negative half tube}. 
The former one comes from the elliptic regularity and the following estimate:
For $t>1$,
\[ \int_{p^{-1}(t,+\infty)}|a|^2 d\vol \leq e^{-2\alpha t}\int_{p^{-1}(t,+\infty)}e^{2\alpha p(x)}|a(x)|^2 d\vol(x) 
   \leq \norm{a}^2_{L^{2,\alpha}}e^{-2\alpha t}.\]

\textbf{Step 3}: From Step 1 and $H^1_{dR}(\hat{Y}) =0$, there is a smooth function $f$ on $\hat{Y}$ satisfying $a=df$.
From Step 2, the limits $f(+\infty) := \lim_{t\to +\infty} f(t, \theta)$ and 
$f(-\infty) := \lim_{t\to -\infty}f(t,\theta)$ exist and independent of $\theta\in S^3$.
In particular $f$ is bounded. We can assume $f(+\infty) =0$. Then for $t>1$
\[ f(t, \theta) = -\int_{t}^\infty \frac{\partial f}{\partial s}(s, \theta) ds.\]
Since $|\partial f/\partial s|\leq |a|\leq \const_{a,\alpha}e^{-\alpha t}$ for $t>1$ (Step 2), 
\begin{equation} \label{eq: decay of f over the positive end}
 |f|\leq \const_{a, \alpha}\cdot e^{-\alpha t} \quad (t>1).
\end{equation}
Let $\beta_R$ be the cut-off function used in Step 1.
Since $e^{-2\alpha t}d^*(e^{2\alpha t}a) = d^{*,\alpha}a =0$,
\begin{equation*}
 \begin{split}
  0 = \int e^{2\alpha t}\langle \beta_R f, d^{*,\alpha}a\rangle d\vol&= 
   \int e^{2\alpha t}\langle d(\beta_R f), a\rangle  d\vol \\
   &= \int e^{2\alpha t}f\langle d\beta_R, a\rangle d\vol + \int e^{2\alpha t}\beta_R |a|^2 d\vol.
 \end{split}
\end{equation*}
Hence 
\begin{equation} \label{eq: weighted norm of beta_T |a|^2}
   \int e^{2\alpha t}\beta_R |a|^2 d\vol \leq  
   \frac{2}{R}\int_{\supp (d\beta_R)} e^{2\alpha t}|f||a|d\vol.
\end{equation}
We have $\supp (d\beta_R)\subset p^{-1}(-2R,-R)\cup p^{-1}(R,2R)$. 
Since $|f|$ and $|a|$ are bounded,
\[ \int_{p^{-1}(-2R,-R)} e^{2\alpha t}|f||a|d\vol \to 0 \quad (R\to +\infty).\]
On the other hand, by the above (\ref{eq: decay of f over the positive end})
\begin{equation*}
 \begin{split}
 &\frac{2}{R}\int_{p^{-1}(R,2R)} e^{2\alpha t}|f||a|d\vol \leq \frac{\const_{a,\alpha}}{R}\int_{p^{-1}(R,2R)}e^{\alpha t}|a|d\vol \\
 &\leq \frac{\const_{a,\alpha}}{R}\sqrt{\vol((R, 2R)\times S^3)}\sqrt{\int_{p^{-1}(R,2R)}e^{2\alpha t}|a|^2d\vol}
 \leq \const_{a,\alpha} \norm{a}_{L^{2,\alpha}}/\sqrt{R}.
 \end{split}
\end{equation*}
This goes to $0$ as $R\to +\infty$.
From (\ref{eq: weighted norm of beta_T |a|^2}),
\[ \int e^{2\alpha t}|a|^2 =0.\]
Thus $a=0$.
\end{proof}
\begin{lemma} \label{lemma: quantitative version of the first cohomology is zero}
For $a\in L^{2,\alpha}_1(\hat{Y}, \Lambda^1)$, 
\[ \norm{a}_{L^{2,\alpha}_1(\hat{Y})}\leq \const_\alpha \norm{D^\alpha a}_{L^{2,\alpha}(\hat{Y})}.\]
\end{lemma}
\begin{proof}
Set $U:=p^{-1}(-2,2)\subset \hat{Y}$.
By using Lemma \ref{lemma: L^2,alpha estimate}, for all $a\in L^{2,\alpha}_1(\hat{Y})$
\begin{equation} \label{eq: inner and outer estimate}
 \norm{a}_{L_1^{2,\alpha}(\hat{Y})} \leq 
   \const_\alpha (\norm{a}_{L^{2}(U)} + \norm{D^\alpha a}_{L^{2,\alpha}(\hat{Y})}).
\end{equation}
We want to show $\norm{a}_{L^{2}(U)}\leq \const_\alpha \norm{D^\alpha a}_{L^{2,\alpha}(\hat{Y})}$.
Suppose on the contrary there exist a sequence $a_n$ $(n\geq 1)$ in $L^{2,\alpha}_1(\hat{Y}, \Lambda^1)$ such that 
\[ 1= \norm{a_n}_{L^{2}(U)} > 
    n\norm{D^\alpha a_n}_{L^{2,\alpha}(\hat{Y})}.\]
From the above (\ref{eq: inner and outer estimate}), $\{a_n\}$ is bounded in $L^{2,\alpha}_1(\hat{Y})$.
Hence, if we take a subsequence (also denoted by $a_n$), 
the sequence $a_n$ weakly converges to some $a$ in $L^{2,\alpha}_1(\hat{Y})$.
We have $D^\alpha a=0$.
Hence Lemma \ref{lemma: the first cohomology is zero} implies $a=0$.
By Rellich's lemma, $a_n$ strongly converges to $0$ in $L^{2}(U)$.
(Note that $U$ is pre-compact.)
This contradicts $\norm{a_n}_{L^{2}(U)}=1$.
\end{proof}

\subsection{Preliminary results over $X=Y^{\sharp \mathbb{Z}}$} \label{subsection: preliminary results over X=Y^sharp}
Recall that $X=Y^{\sharp \mathbb{Z}}$ has the periodic metric $g_0$ 
which is compatible with the given metric $h$ over every $Y^{(n)}$
($n\in \mathbb{Z}$), and that $g_0$ depends on the parameter $T>2$.
We define the Sobolev norm $\norm{\cdot}_{L^2_k}$ over $X$ as in (\ref{eq: sobolev norm}) 
by using the metric $g_0$ and its Levi-Civita connection.
We define the weighted Sobolev norm by $\norm{u}_{L^{2,\alpha}_k} := \norm{e^{\alpha t}u}_{L^{2,\alpha}_k}$ where 
$t=q(x)$ $(x\in X)$.
For a $1$-form $a$ over $X$ we set $D^\alpha a := -d^{*, \alpha} a + d^+ a = -e^{-2\alpha t}d^*(e^{2\alpha t}a) + d^+a$.
\begin{lemma} \label{lemma: quantitative version of the first cohomology of X is zero}
There exists $T_\alpha>2$ such that if $T\geq T_\alpha$ then for any $a\in L^{2,\alpha}_1(X, \Lambda^1)$ we have 
\[ \norm{a}_{L^{2,\alpha}_1(X)} \leq \const_\alpha \norm{D^\alpha a}_{L^{2,\alpha}(X)}.\]
The important point is that $T_\alpha$ depends only on $\alpha$.
\end{lemma}
\begin{proof}
Let $\beta^{(n)}$ be a smooth function on $X$ such that $0\leq \beta^{(n)}\leq 1$, 
$\supp \beta^{(n)}\subset Y^{(n)}_T = q^{-1}((n-1)T+1,(n+1)T-1)$, 
$\beta^{(n)} = 1$ over $q^{-1}((n-1/2)T,(n+1/2)T)$ and $|d\beta^{(n)}|\leq 3/T$.
Since $t=q(x) = p^{(n)}(x) + nT$ over $Y^{(n)}_T$, 
by applying Lemma \ref{lemma: quantitative version of the first cohomology is zero} to $\beta^{(n)} a$, we get 
\begin{equation*}
 \begin{split}
  \norm{\beta^{(n)}a}_{L^{2,\alpha}_1(X)} &= e^{\alpha nT}\norm{e^{\alpha p^{(n)}(x)}\beta^{(n)}a}_{L^2_1(Y^{(n)}_T)} \\
  &\leq \const_\alpha \cdot e^{\alpha nT}\norm{e^{\alpha p^{(n)}(x)}D^\alpha (\beta^{(n)}a)}_{L^2(Y^{(n)}_T)} \\
  &= \const_\alpha \norm{D^\alpha (\beta^{(n)}a)}_{L^{2,\alpha}(X)}\\
    &\leq \frac{\const_\alpha}{T}\norm{a}_{L^{2,\alpha}(Y^{(n)}_T)} 
    + \const_\alpha \norm{D^\alpha a}_{L^{2,\alpha}(Y^{(n)}_T)}.
 \end{split}
\end{equation*}
Then 
\begin{equation*}
 \begin{split}
 \norm{a}_{L^{2,\alpha}_1(X)}^2 &\leq \sum_{n\in \mathbb{Z}} \norm{\beta^{(n)}a}_{L^{2,\alpha}_1(X)}^2 \\
 &\leq \frac{\const_\alpha}{T^2} \sum_{n\in \mathbb{Z}} \norm{a}_{L^{2,\alpha}(Y^{(n)}_T)}^2
  + \const_\alpha \sum_{n\in \mathbb{Z}}\norm{D^\alpha a}_{L^{2,\alpha}(Y^{(n)}_T)}^2 \\ 
 &\leq \frac{\const_\alpha}{T^2}\norm{a}_{L^{2,\alpha}(X)}^2  + \const_\alpha \norm{D^\alpha a}_{L^{2,\alpha}(X)}^2.
 \end{split}
\end{equation*}
If $T\gg 1$, then 
\[ \norm{a}_{L^{2,\alpha}_1(X)}^2 \leq \const_\alpha \norm{D^\alpha a}_{L^{2,\alpha}(X)}^2.\]
\end{proof}
For a $1$-form $a$ on $X$ we set $\mathcal{D}a:=-d^*a+d^+a$.
Its formal adjoint $\mathcal{D}^*$ is given by 
$\mathcal{D}^*(u,\xi) = -du + d^*\xi= -du-*d\xi$ for $(u,\xi)\in \Omega^0\oplus\Omega^+$.
We consider $\mathcal{D}$ as an unbounded operator from $L^2(X,\Lambda^1)$ to $L^2(X,\Lambda^0\oplus\Lambda^+)$.

The additive Lie group $\mathbb{Z}$ naturally acts on $X=Y^{\sharp\mathbb{Z}}$.
Set $Y^+ :=X/\mathbb{Z}$.
We have $b_1(Y^+) =1$ and $b_+(Y^+)=b_+(Y)$.
The operator $\mathcal{D}$ is preserved by the $\mathbb{Z}$-action, and its quotient is equal to the operator 
$-d^*+d^+ :\Omega^1_{Y^+}\to \Omega^0_{Y^+}\oplus\Omega^+_{Y^+}$ on $Y^+$.
Then we can apply Atiyah's $\Gamma$-index theorem (Atiyah \cite{Atiyah}, Roe \cite[Chapter 13]{Roe}) to $\mathcal{D}$ and get 
\[ \mathrm{ind}_\mathbb{Z} \mathcal{D} = \mathrm{ind} (-d^*+d^+ :\Omega^1_{Y^+}\to \Omega^0_{Y^+}\oplus\Omega^+_{Y^+})
   = -1+b_1(Y^+)-b_+(Y^+) = -b_+(Y).\]
Here $\mathrm{ind}_\mathbb{Z}\mathcal{D}$ is the $\Gamma$-index of $\mathcal{D}$ $(\Gamma =\mathbb{Z})$.

The above implies that if $b_+(Y)\geq 1$ then $\mathrm{Ker}\mathcal{D}^*\subset L^2(X,\Lambda^0\oplus\Lambda^+)$
is infinite dimensional.
Suppose $\rho=(u,\xi)\in L^2(X,\Lambda^0\oplus\Lambda^+)$ satisfies $\mathcal{D}^*\rho=-du+d^*\xi=0$ as a distribution.
By the elliptic regularity, $\rho$ is smooth, and
for each $n\in \mathbb{Z}$
\[ \norm{\rho}_{L^2_1(q^{-1}(-(n-1/2)T,(n+1/2)T))}\leq \const_T \norm{\rho}_{L^2(Y^{(n)}_T)}.\]
Here $\const_T$ is independent of $n\in \mathbb{Z}$.
Hence $\norm{\rho}_{L^2_1(X)}\leq \const_T \norm{\rho}_{L^2(X)} <+\infty$, and
 $\rho\in L^2_1(X)$.
In particular $u, \xi\in L^2_1(X)$ and hence $\langle du,d^*\xi\rangle_{L^2}=0$.
Then 
\[ 0=\langle \mathcal{D}^*\rho,du\rangle_{L^2} = -\norm{du}_{L^2}.\]
So $du=0$.
This means that $u$ is constant. But $u\in L^2$. Hence $u=0$.
Therefore $d^*\xi=0$.
Thus we get the following result.
\begin{lemma} \label{lemma: the space of harmonic self-dual forms is infinite dimensional}
Suppose $b_+(Y)\geq 1$.
The space of $\xi\in L^2_1(X,\Lambda^+)$ satisfying $d^*\xi=0$ is infinite dimensional.
\end{lemma}
Take and fix a smooth function $|\cdot|':\mathbb{R}\to \mathbb{R}$ satisfying $|t|' = |t|$ for $|t|\geq 1$.
For $0<|\alpha|<1$, set $W(x):=e^{\alpha|q(x)|'}$ for $x\in X$.
Hence $W$ is a positive smooth function on $X$ satisfying $W(x) = e^{\alpha|q(x)|}$ for $|q(x)| \geq 1$.
For a section $\eta$ of $\Lambda^i$ $(i\geq 0)$ we set 
$\norm{\eta}_{L^{2,W}_k(X)} := \norm{W\eta}_{L^2_k(X)}$.
For a self-dual form $\eta$ over $X$, we set 
$d^{*, W}\eta := -W^{-2}*d(W^2\eta)$.
If $a\in \Omega^1_X$ and $\eta\in \Omega^+_X$ have compact supports, then 
$\int_X W^2\langle da, \eta\rangle d\vol = \int_X W^2\langle a, d^{*,W}\eta\rangle d\vol$.
\begin{lemma} \label{lemma: infinite dimensionality of weighted harminic self-dual forms}
Suppose $b_+(Y)\geq 1$ and $\alpha>0$.
Then the space of $\eta\in L^{2,W}_1(X,\Lambda^+)$ satisfying $d^{*,W}\eta=0$ is infinite dimensional.
Moreover it is closed in $L^{2,W}(X, \Lambda^+)$.
\end{lemma}
\begin{proof}
Suppose that $\xi\in L^2_1(X,\Lambda^+)$ satisfies $d^*\xi=0$. Set $\eta:=W^{-2}\xi$.
Then $d^{*,W}\eta=0$ and 
$\norm{\eta}_{L^{2,W}_1(X)} = \norm{W^{-1}\xi}_{L^2_1(X)} <\infty$ from $\alpha>0$.
Thus Lemma \ref{lemma: the space of harmonic self-dual forms is infinite dimensional} implies 
the first statement.

In order to prove the closedness of $\mathrm{Ker}(d^{*,W}) \subset L^{2,W}_1(X, \Lambda^+)$ in $L^{2,W}(X, \Lambda^+)$, 
it is enough to show that if $\eta\in L^{2,W}(X,\Lambda^+)$ satisfies $d^{*,W}\eta=0$ (as a distribution) then 
$\eta\in L^{2,W}_1(X,\Lambda^+)$.
$\eta$ is smooth by the (local) elliptic regularity.
The differential operator $d^{*,W}$ on $Y^{(n)}_T$ $(n>0)$ are naturally isomorphic to each other.
The same statement also hold for $n<0$.
Hence, by the elliptic regularity, 
\begin{equation*}
 \begin{split}
 \norm{W\eta}_{L^2_1(q^{-1}((n-1/2)T, (n+1/2)T))} &\leq \const_{T,\alpha}\cdot (\norm{W\eta}_{L^2(Y^{(n)}_T)} 
   + \norm{d^{*,W}(W\eta)}_{L^2(Y^{(n)}_T)})\\
 &\leq \const_{T,\alpha} \norm{W\eta}_{L^2(Y^{(n)}_T)}.
 \end{split}
\end{equation*}
Here $\const_{T,\alpha}$ are independent of $n\in \mathbb{Z}$.
Thus $\norm{\eta}_{L^{2,W}_1(X)}\leq \const_{T,\alpha}\norm{\eta}_{L^{2,W}(X)} < \infty$.
\end{proof}
\begin{lemma} \label{lemma: techniacl lemma for infinite dimensionality of harmonic self-dual forms}
Suppose $b_+(Y)\geq 1$ and $\alpha>0$.
For any $\varepsilon>0$ and any pre-compact open set $U\subset X$, there is $\eta\in L^{2,W}_1(X,\Lambda^+)$ such that 
$\eta=0$ over $U$ and
\[ \norm{d^{*,W}\eta}_{L^{2,W}(X)} < \varepsilon \norm{\eta}_{L^{2,W}(X)}.\]
\end{lemma}
\begin{proof}
First we prove the following statement:
For any $\varepsilon>0$ and any pre-compact open set $U\subset X$ there exists 
$\eta \in L^{2,W}_1(X,\Lambda^+)$ satisfying $d^{*,W}\eta=0$ and 
$\norm{\eta}_{L^{2,W}(U)} < \varepsilon \norm{\eta}_{L^{2,W}(X)}$.
Suppose that this statement does not hold.
Then there are $\varepsilon>0$ and a pre-compact open set $U\subset X$ such that all
$\eta\in \mathrm{Ker}(d^{*,W})\subset L^{2,W}_1(X,\Lambda^+)$ satisfies
\[ \norm{\eta}_{L^{2,W}(U)}\geq \varepsilon\norm{\eta}_{L^{2,W}(X)}.\]
$\mathrm{Ker}(d^{*,W})$ is an infinite dimensional closed subspace in $L^{2,W}(X, \Lambda^+)$
(Lemma \ref{lemma: infinite dimensionality of weighted harminic self-dual forms}).
Let $\{\eta_n\}_{n\geq 1}$ be a complete orthonormal basis of $\mathrm{Ker}(d^{*,W})$
with respect to the inner product of $L^{2,W}(X,\Lambda^+)$.
They satisfies 
\[ \norm{\eta_n}_{L^{2,W}(U)} \geq \varepsilon.\]
The sequence $\eta_n$ weakly converges to $0$ in $L^{2,W}(X)$, and hence 
$\eta_n|_{U}$ weakly converges to $0$ in $L^{2,W}(U)$.
Then, by the elliptic regularity and Rellich's lemma, a subsequence of $\eta_n|_U$
strongly converges to $0$ in $L^{2,W}(U)$.
But this contradicts $\norm{\eta_n}_{L^{2,W}(U)}\geq \varepsilon$.

Next take a pre-compact open set $V\subset X$ satisfying the following: 
$U\subset V$ and there exists a smooth 
function $\beta$ such that $0\leq \beta\leq 1$, $\beta=0$ on $U$, $\beta=1$ on $X\setminus V$,
$\supp(d\beta)\subset V$, and $|d\beta|\leq \varepsilon$. 
By the previous argument
there exists $\eta\in L^{2,W}_1(X,\Lambda^+)$ satisfying 
$d^{*, W}\eta=0$ and $\norm{\eta}_{L^{2,W}(V)} < (1/3)\norm{\eta}_{L^{2,W}(X)}$.
Then $\norm{\beta\eta}_{L^{2,W}(X)}> (2/3)\norm{\eta}_{L^{2,W}(X)}$.
Since $d^{*, W}(\beta\eta) = -*(d\beta\wedge \eta)$ is supported in $V$, 
\[ \norm{d^{*,W}(\beta\eta)}_{L^{2,W}(X)}\leq \varepsilon\norm{\eta}_{L^{2,W}(V)}
   <(\varepsilon/3)\norm{\eta}_{L^{2,W}(X)} < (\varepsilon/2)\norm{\beta\eta}_{L^{2,W}(X)}.\]
Hence $\beta\eta\in L^{2,W}_1(X,\Lambda^+)$ satisfies $\beta\eta=0$ over $U$ and 
$\norm{d^{*,W}(\beta\eta)}_{L^{2,W}(X)} < \varepsilon\norm{\beta\eta}_{L^{2,W}(X)}$.
\end{proof}

\section{Linear theory}  \label{section: linear theory}
In this section we always assume $0<\alpha<1$ and 
\begin{equation} \label{eq: condition on T for the linear theory}
 T\geq \max(T_\alpha, T_{-\alpha}).
\end{equation}
Here $T_\alpha$ and $T_{-\alpha}$ are the positive constants introduced in 
Lemma \ref{lemma: quantitative version of the first cohomology of X is zero}.
(Recall that they depend only on $\alpha$.)
The purpose of this section is to prove several basic properties of the linear operators
$D_A$ and $D_A'$ introduced below.
The constants introduced in this section often depend on several parameters
($\alpha$, $T$, $A_0$, $A$, $\mu$).
But we usually don't explicitly write their dependence on parameters
unless it causes a confusion.

\subsection{The image of $D_A$ is closed} \label{subsection: the image of D_A is closed}
Let $E = X\times SU(2)$ be the product principal $SU(2)$-bundle over $X$, and $A_0$ 
be an adapted connection on $E$ (see Definition \ref{def: adapted connection}).
Let $W = e^{\alpha |q(x)|'}$ be the weight function on $X$ introduced 
in Section \ref{subsection: preliminary results over X=Y^sharp}.
For a section $u$ of $\Lambda^i(\ad E)$ $(i\geq 0)$, we define the Sobolev norm $\norm{u}_{L^2_k}$ by using
the periodic metric $g_0$ and the connection $A_0$.
We define the weighted Sobolev norm by $\norm{u}_{L^{2,W}_k} := \norm{W u}_{L^{2}_k}$. 

Let $\Lambda^{+}$ and $\Lambda^-$ be the bundles of self-dual and 
anti-self-dual forms (with respect to the metric $g_0$) on $X$,
and $\mu:\Lambda^-\to \Lambda^+$ be a smooth bundle map. 
We assume $|\mu_x| <1$ for all $x\in X$ (i.e. $|\mu(\omega)|<|\omega|$ for all non-zero $\omega\in \Lambda^-$ where 
the norm $|\cdot|$ is defined by the metric $g_0$).
Moreover we assume that $\mu$ is compact supported.
Hence $\mu$ corresponds to a conformal structure on $X$ which coincides with $[g_0]$ outside a compact set
(see Section \ref{subsection: anti-self-duality and conformal structure}).

We define $\mathcal{A} = \mathcal{A}_{A_0}$ as the space of $L^{2,W}_3$-connections (with respect to $A_0$) on $E$:
\[ \mathcal{A} :=\{A_0+a|\, a\in L^{2,W}_3(X, \Lambda^1(\ad E))\}.\]
(Recall that the connection $A_0$ is used in the definition of the weighted Sobolev space $L_3^{2,W}(X,\Lambda^1(\ad E))$.)
We will need the following multiplication rule:
If $k\geq 3$ and $k\geq l$, then $L^{2,W}_k\times L^{2,W}_l\to L^{2,W}_l$, i.e.  
for $f_1\in L^{2,W}_{k}$ and $f_2\in L^{2,W}_l$ $(k\geq 3,\, k\geq l\geq 0)$
\begin{equation} \label{eq: multiplication rule}
 \norm{f_1 f_2}_{L^{2,W}_l} \leq \const \norm{f_1}_{L^{2,W}_{k}}\norm{f_2}_{L^{2,W}_{l}}.
\end{equation}
In particular, for $A=A_0+a\in \mathcal{A}$, we have 
$F(A)=F(A_0)+d_{A_0}a+a\wedge a\in L^{2,W}_2$.
For $b\in \Omega^1(\ad E)$ over $X$, we set 
\[ D_A b := -d_A^{*,W}b+(d_A^{+}-\mu d_A^-)b = -W^{-2}d_A^{*}(W^2 b) + (d_A^{+}-\mu d_A^-)b.\]
Here $d_A^*b = -*d_A(*b)$ and $d_A^{\pm}=\frac{1}{2}(1\pm *)d_A$.
($*$ is the Hodge star defined by the metric $g_0$.)
$D_{A}$ is an elliptic differential operator since we assume $|\mu_x|<1$ for all $x\in X$.
Rigorously speaking,
we should use the notation $D^{\mu,W}_{A}$ instead of $D_A$.
But here we use the above notation for simplicity.
We have 
\begin{equation} \label{eq: D_A and D_A_0}
 D_Ab = D_{A_0}b + *[a\wedge *b] + [a\wedge b]^+ -\mu([a\wedge b]^-).
\end{equation}
From this and the above (\ref{eq: multiplication rule}), the map 
$D_A :L^{2,W}_{k+1}(X,\Lambda^1(\ad E))\to L^{2,W}_k(X,(\Lambda^0\oplus\Lambda^+)(\ad E))$
$(0\leq k\leq 3)$ becomes a bounded linear map.

Let $r$ be a positive integer such that $q^{-1}(-rT,rT)$ contains the supports of $F(A_0)$ and $\mu$.
Set $U:=q^{-1}(-(r+5/2)T,(r+5/2)T)$.
\begin{lemma}  \label{lemma: elliptic estimate for D_A} 
(i)
For any 
$b\in L^{2,w}_{k+1}(X, \Lambda^1(\ad E))$ $(k\geq 0)$ we have 
\begin{equation} \label{eq: the crucial estimate for the closedness}
 \norm{b}_{L^{2,W}_{k+1}(X)}\leq \const\, (\norm{b}_{L^{2}(U)} + 
   \norm{D_{A_0}b}_{L_k^{2,W}(X)}).
\end{equation}
Here $\const$ is a positive constant independent of $b$.
(We will usually omit this kind of obvious remark below.)

\noindent
(ii) 
For any $A=A_0+a\in \mathcal{A}$, there is a pre-compact open set $U_{A}\subset X$ 
(which depends on $\mu,\alpha,T,A_0,A$) such that 
for any $b\in L^{2,W}_{k+1}(X,\Lambda^1(\ad E))$ $(0\leq k\leq 3)$ 
\begin{equation} \label{eq: elliptic estimate for D_A of A=A_0+a}
 \norm{b}_{L^{2,W}_{k+1}(X)}\leq \const \, (\norm{b}_{L^{2}(U_{A})} + \norm{D_Ab}_{L^{2,W}_k(X)}).
\end{equation}
\end{lemma}
\begin{proof}
(i)
We first consider the case $k=0$.
From Lemma \ref{lemma: quantitative version of the first cohomology of X is zero} and 
the condition (\ref{eq: condition on T for the linear theory}),
for any $b_1\in L^{2,\alpha}_1(X,\Lambda^1)$ and $b_2\in L^{2,-\alpha}_1(X,\Lambda^1)$
\begin{gather}
 \norm{b_1}_{L^{2,\alpha}_1} \leq \const \norm{D^{\alpha} b_1}_{L^{2,\alpha}} \label{eq: posotive alpha case},\\
 \norm{b_2}_{L^{2,-\alpha}_1}\leq \const \norm{D^{-\alpha} b_2}_{L^{2,-\alpha}} \label{eq: negative alpha case}.
\end{gather}
Let $b \in L^{2,W}_1(X,\Lambda^1(\ad E))$.
Let $\beta$ be a smooth function on $X$ such that $\beta = 0$ on $t\leq (r+1/2)T$ and $\beta = 1$ on $t\geq (r+1)T$
$(t=q(x))$.
Recall that $\supp (\mu)$ and $\supp (F_{A_0})$ are contained in $q^{-1}(-rT,rT)$ and that 
$W=e^{\alpha t}$ for $t\geq 1$.
By applying the above (\ref{eq: posotive alpha case}) to $\beta a$, we get 
\begin{equation} \label{eq: estimate for beta a}
 \norm{\beta b}_{L_1^{2,W}(X)} \leq \const \norm{D_{A_0} (\beta b)}_{L^{2,W}(X)}
   \leq \const \, (\norm{b}_{L^{2}(U)} + \norm{D_{A_0}b}_{L^{2,W}(X)}).
\end{equation} 
Let $\beta'$ be a smooth function on $X$ such that $\beta'=0$ on $t\geq -(r+1/2)T$ and $\beta'=1$ on $t\leq -(r+1)T$.
By applying (\ref{eq: negative alpha case}) to $\beta' b$, we get 
\begin{equation} \label{eq: estimate for beta' a}
 \norm{\beta' b}_{L^{2,W}_1(X)}\leq \const \norm{D_{A_0}(\beta' b)}_{L^{2,W}(X)}
   \leq \const \, (\norm{b}_{L^{2}(U)} + \norm{D_{A_0}b}_{L^{2,W}(X)}) .
\end{equation}
From the elliptic regularity, 
\[ \norm{b}_{L^{2,W}_1(q^{-1}(-(r+3/2)T,(r+3/2)T))} \leq \const\,
   (\norm{b}_{L^{2}(U)} + \norm{D_{A_0}b}_{L^{2}(U)}).\]
This estimate and the above (\ref{eq: estimate for beta a}) and (\ref{eq: estimate for beta' a}) imply 
\begin{equation}\label{eq: the crucial estimate for the closedness for k=0}
 \norm{b}_{L^{2,W}_1(X)}\leq \const \, (\norm{b}_{L^{2}(U)}+\norm{D_{A_0}b}_{L^{2,W}(X)}).
\end{equation}

Next let $b\in L^{2,W}_{k+1}(X,\Lambda^1(\ad E))$.
From the elliptic regularity, for any $n\in \mathbb{Z}$
\begin{equation*}
 \begin{split}
 \norm{b}_{L^{2,W}_{k+1}(q^{-1}((n-1/2)T,\,(n+1/2)T))} &= \norm{Wb}_{L^{2}_{k+1}(q^{-1}((n-1/2)T,\,(n+1/2)T))} \\
 &\leq \const \, (\norm{Wb}_{L^2_k(Y^{(n)}_T)} + \norm{D_{A_0}(Wb)}_{L^2_k(Y^{(n)}_T)}) \\
 &\leq \const \, (\norm{b}_{L^{2,W}_k(Y^{(n)}_T)} + \norm{D_{A_0}b}_{L^{2,W}_k(Y^{(n)}_T)}).
 \end{split}
\end{equation*}
The above two ``const'' are independent of $n\in \mathbb{Z}$.
Therefore
\begin{equation} \label{eq: ellitptic estimate for D_{A_0}}
 \norm{b}_{L^{2,W}_{k+1}(X)}\leq \const\, (\norm{b}_{L^{2,W}_k(X)}+\norm{D_{A_0}b}_{L^{2,W}_k(X)}).
\end{equation}
By using this estimate and the above (\ref{eq: the crucial estimate for the closedness for k=0}), 
we can inductively prove (\ref{eq: the crucial estimate for the closedness}).

(ii) 
From (i)
\begin{equation} \label{eq: the crucial estimate for the closedness, re-visited}
 \norm{b}_{L^{2,W}_{k+1}(X)}\leq C(\norm{b}_{L^{2}(U)} + \norm{D_{A_0} b}_{L^{2,W}_k(X)}),
\end{equation}
where the positive constant $C$ depends on $\mu,\alpha,T,A_0$.
Take $\varepsilon>0$ so that $C\varepsilon<1$.
From (\ref{eq: multiplication rule}), (\ref{eq: D_A and D_A_0}) and $a\in L^{2,W}_3$,
there is a positive integer $r_A>r$ ($U_A := q^{-1}(-(r_A+5/2)T,(r_A+5/2)T)\supset U$) such that 
\[ \norm{D_A b-D_{A_0}b}_{L^{2,W}_k(X\setminus \overline{U_{A}})} 
 \leq \varepsilon\norm{b}_{L^{2,W}_{k}(X)} \quad (0\leq k\leq 3).\]
On the other hand
\[ \norm{D_A b- D_{A_0}b}_{L^{2,W}_k(U_{A})}\leq \const \norm{b}_{L^{2}_k(U_A)}.\]
Therefore, from (\ref{eq: the crucial estimate for the closedness, re-visited}),
\[ \norm{b}_{L^{2,W}_{k+1}(X)}\leq \const \, (\norm{b}_{L^{2}_k(U_A)}+\norm{D_Ab}_{L^{2,W}_k(X)})
   + C\varepsilon \norm{b}_{L^{2,W}_{k}(X)}.\]
Since $C\varepsilon<1$, we get 
\[ \norm{b}_{L^{2,W}_{k+1}(X)} \leq \const \, (\norm{b}_{L^{2}_k(U_A)} + \norm{D_Ab}_{L^{2,W}_k(X)}).\]
By the induction on $k$, we get (\ref{eq: elliptic estimate for D_A of A=A_0+a}).
\end{proof}
\begin{proposition}\label{prop: closedness of the image of the linearlization}
Let $A\in \mathcal{A}$.
If $b\in L^{2,W}(X,\Lambda^1(\ad E))$ satisfies $D_Ab=0$ as a distribution, then $b\in L^{2,W}_4(X)$.
Let $0\leq k\leq 3$.
The kernel of the map $D_A:L^{2,W}_{k+1}(X,\Lambda^1(\ad E))\to L^{2,W}_k(X,(\Lambda^0\oplus\Lambda^+)(\ad E))$ 
is of finite dimension, and 
the image $D_A(L^{2,W}_{k+1}(X, \Lambda^1(ad E)))$ is closed in $L^{2,W}_k(X, (\Lambda^0\oplus \Lambda^+)(\ad E))$.
\end{proposition}
\begin{proof}
The first regularity statement ($D_Ab=0 \Rightarrow b\in L^{2,W}_4$)
follows from Lemma \ref{lemma: elliptic estimate for D_A} (ii).
Let $\mathrm{Ker}D_A$ be the space of $b\in L^{2,W}_4(X, \Lambda^1(\ad E))$ satisfying 
$D_Ab=0$.
For any $b\in \ker D_A$, $\norm{b}_{L^{2,W}_{4}(X)}\leq \const \norm{b}_{L^{2}(U_A)}$
by Lemma \ref{lemma: elliptic estimate for D_A} (ii).
Here $U_A$ is a pre-compact open set.
Then the standard argument using Rellich's lemma shows the finite dimensionality of $\ker D_A$.
\begin{sublemma}\label{sublemma: elliptic estimate for L^{2,w}-orthogonal element}
If $b\in L^{2,W}_{k+1}(X,\Lambda^1(\ad E))$ $(0\leq k\leq 3)$ is $L^{2,W}$-orthogonal to $\mathrm{Ker}D_A$
(i.e. $\int_X W^2\langle b, \beta\rangle d\vol =0$ for all 
$\beta\in \mathrm{Ker}D_A$)
then 
\[ \norm{b}_{L^{2,W}_{k+1}(X)}\leq \const \norm{D_Ab}_{L^{2,W}_k(X)}.\]
\end{sublemma}
\begin{proof}
It is enough to prove  
$\norm{b}_{L^{2}(U_A)}\leq \const \norm{D_Ab}_{L^{2,W}(X)}$.
Since $U_A$ is pre-compact, this follows from the standard argument using
Lemma \ref{lemma: elliptic estimate for D_A} (ii) and Rellich's lemma.
\end{proof}
Let $H\subset L^{2,W}_{k+1}(X,\Lambda^1(\ad E))$ be the $L^{2,W}$-orthogonal complement of $\ker D_A$.
Then Sublemma \ref{sublemma: elliptic estimate for L^{2,w}-orthogonal element} shows that
$\mathrm{image} (D_A) = D_A(H)$ is a closed subspace in $L^{2,W}_k(X, \Lambda^1(\ad E))$.
\end{proof}

\subsection{The kernel of $D_A'$ is infinite dimensional} \label{subsection: ker D_A' is infinite dimensional}
For $\mu:\Lambda^-\to \Lambda^+$ we define $\mu^*:\Lambda^+\to \Lambda^-$ by 
\[ \mu(\xi)\wedge \eta = \xi\wedge \mu^*(\eta) \quad (\xi\in \Lambda^-, \eta\in \Lambda^+).\]
Let $A = A_0+a\in \mathcal{A}$.
For $\omega\in \Omega^2(\ad E)$, we set 
$d_A^{*,W}\omega = -W^{-2}*d_A(*W^2\omega)$.
If $b\in \Omega^1(\ad E)$ and $\omega\in \Omega^2(\ad E)$ have compact supports, then 
$\int_X W^2\langle d_Ab, \omega\rangle d\vol = \int_X W^2\langle b, d_A^{*,W}\omega\rangle d\vol$.
For $\rho=(u, \eta)\in \Omega^0(\ad E)\oplus \Omega^+(\ad E)$, 
we set
\[ D'_A\rho := -d_Au+d_A^{*,W}(1+\mu^*)\eta = -d_Au - W^{-2}* d_A(W^2(1-\mu^*)\eta).\]
$D_A'$ is an elliptic differential operator. 
If $b\in \Omega^1(\ad E)$ and $\rho\in \Omega^0(\ad E)\oplus \Omega^+(\ad E)$ have compact supports, then 
$\int_X W^2\langle D_A b, \rho\rangle d\vol = \int_X W^2\langle b, D_A'\rho\rangle d\vol$.
We have 
\begin{equation} \label{eq: D'_A and D'_{A_0}}
 D_A'(u,\eta) = D_{A_0}'(u,\eta) -[a,u]-*[a\wedge (1-\mu^*)\eta].
\end{equation}
From the multiplication rule (\ref{eq: multiplication rule}),
$D'_A$ defines a bounded linear map $D_A':L^{2,W}_{k+1}(X,(\Lambda^0 \oplus \Lambda^+)(\ad E))\to L^{2,W}_k(X,\Lambda^1(\ad E))$
for $0\leq k\leq 3$.
\begin{lemma}  \label{lemma: elliptic estimate for D'_A}
For any $\rho\in L^{2,W}_{k+1}(X,(\Lambda^0\oplus\Lambda^+)(\ad E))$ $(0\leq k\leq 3)$,
\[ \norm{\rho}_{L^{2,W}_{k+1}(X)}\leq \const \, (\norm{\rho}_{L^{2,W}(X)} + \norm{D'_A\rho}_{L^{2,W}_k(X)}).\]
Hence if $\rho\in L^{2,W}(X)$ satisfies $D_A'\rho=0$ as a distribution, then $\rho\in L^{2,W}_4(X)$.
\end{lemma}
\begin{proof}
In the same way as in the proof of the estimate (\ref{eq: ellitptic estimate for D_{A_0}}), we get 
\[ \norm{\rho}_{L^{2,W}_{k+1}(X)}\leq \const \,
  (\norm{\rho}_{L^{2,W}(X)}+\norm{D'_{A_0}\rho}_{L^{2,W}_k(X)}).\]
By using the multiplication rule (\ref{eq: multiplication rule}), we get the desired estimate.
The regularity statement easily follows from the above estimate.
\end{proof}
Let $\mathrm{Ker}D_A$ be the space of $b\in L^{2,W}_4(X, \Lambda^1(\ad E))$ satisfying $D_A b=0$, and 
$\mathrm{Ker}D_A'$ be the space of $\rho\in L^{2,W}_4(X, (\Lambda^0\oplus \Lambda^+)(\ad E))$ satisfying 
$D_A' \rho=0$.
\begin{lemma} \label{lemma: Hodge decomposition of Lambda^0 Lambda^+}
Let $A\in \mathcal{A}$ and $0\leq k\leq 3$.

\noindent
(i) We have the following $L^{2,W}$-orthogonal decomposition:
\[L^{2,W}_k(X,(\Lambda^0\oplus\Lambda^+)(\ad E)) = D_A(L_{k+1}^{2,W}(X, \Lambda^1(\ad E)))\oplus \mathrm{Ker}D'_A.\]

\noindent 
(ii)
If $\rho \in L^{2,W}_{k+1}(X,(\Lambda^0\oplus \Lambda^+)(\ad E))$
is $L^{2,W}$-orthogonal to the space $\mathrm{Ker}D'_A$, then 
\[ \norm{\rho}_{L^{2,W}_{k+1}(X)} \leq \const \norm{D'_A\rho}_{L^{2,W}_k(X)}.\] 
Hence $D'_A(L^{2,W}_{k+1}(X,(\Lambda^0\oplus \Lambda^+)(\ad E)))$ is a closed subspace in 
$L^{2,W}_k(X, \Lambda^1(\ad E))$.

\noindent
(iii) We have the following $L^{2,W}$-orthogonal decomposition:
\[ L^{2,W}_k(X,\Lambda^1(\ad E)) = D'_A(L^{2,W}_{k+1}(X,(\Lambda^0\oplus \Lambda^+)(\ad E)))\oplus 
   \mathrm{Ker}D_A.\]

\end{lemma}
\begin{proof}
(i)
$\mathrm{Ker}D_A'$ is closed in $L^{2,W}$.
From Proposition \ref{prop: closedness of the image of the linearlization}, $D_A(L^{2,W}_{k+1})$ is closed in $L^{2,W}_k$,
and it is $L^{2,W}$-orthogonal to $\mathrm{Ker}D'_A$.
If $\rho\in L^{2,W}((\Lambda^0\oplus\Lambda^+)(\ad E))$ is $L^{2,W}$-orthogonal to the space $D_A(L^{2,W}_1)$, then 
$D_A'\rho=0$ as a distribution.
Hence $L^{2,W}=D_A(L^{2,W}_1)\oplus \mathrm{Ker}D_A'$.
By this decomposition,
for $\rho\in L^{2,W}_k((\Lambda^0\oplus\Lambda^+)(\ad E))$,
there are $b\in L^{2,W}_1$ and $\rho'\in \mathrm{Ker}D_A'$ satisfying $\rho = D_Ab+\rho'$.
By Lemma \ref{lemma: elliptic estimate for D'_A}, $\rho'\in L^{2,W}_4$ and hence 
$D_Ab=\rho-\rho'\in L^{2,W}_k$.
Then by Lemma \ref{lemma: elliptic estimate for D_A} (ii), $b\in L^{2,W}_{k+1}$.
This shows $L^{2,W}_k = D_A(L^{2,W}_{k+1})\oplus \mathrm{Ker} D'_A$.

(ii)
By (i), 
there is $b\in L^{2,W}_{k+1}(X,\Lambda^1(\ad E))$ satisfying $\rho = D_Ab$.
We can choose $b$ so that it is $L^{2,W}$-orthogonal to $\mathrm{Ker} D_A$ and that
$\norm{b}_{L^{2,W}}\leq \const \norm{D_Ab}_{L^{2,W}} = \const \norm{\rho}_{L^{2,W}}$
(by Sublemma \ref{sublemma: elliptic estimate for L^{2,w}-orthogonal element}).
Then 
\[ \norm{\rho}_{L^{2,W}}^2=\langle \rho, D_Ab\rangle_{L^{2,W}} 
   = \langle D_A'\rho, b\rangle_{L^{2,W}}\leq \norm{D_A'\rho}_{L^{2,W}}\norm{b}_{L^{2,W}}
   \leq \const \norm{D_A'\rho}_{L^{2,W}}\norm{\rho}_{L^{2,W}}.\]
Thus $\norm{\rho}_{L^{2,W}}\leq \const \norm{D_A'\rho}_{L^{2,W}}$.
Then by using Lemma \ref{lemma: elliptic estimate for D'_A}, we get the desired estimate.

(iii)
$D'_A(L^{2,W}_{k+1})$ is $L^{2,W}$-orthogonal to $\mathrm{Ker}D_A$.
If $b\in L^{2,W}(X,\Lambda^1(\ad E))$ is $L^{2,W}$-orthogonal to the space $D'_A(L^{2,W}_{1})$,
then $D_Ab=0$ as a distribution.
Hence $L^{2,W}=D'_A(L^{2,W}_1)\oplus \mathrm{Ker}D_A$.
From this result, for any $b\in L^{2,W}_k(X,\Lambda^1(\ad E))$, there are 
$\rho\in L^{2,W}_1(X,(\Lambda^0\oplus\Lambda^+)(\ad E))$ and $\beta\in \mathrm{Ker}D_A$
satisfying $b=D'_A\rho + \beta$.
From Lemma \ref{prop: closedness of the image of the linearlization}, $\beta\in L^{2,W}_4$.
Thus $D'_A\rho\in L^{2,W}_k$, and hence (Lemma \ref{lemma: elliptic estimate for D'_A}) 
$\rho\in L^{2,W}_{k+1}$.
\end{proof}
\begin{lemma}\label{lemma: Coulom gauge}
Let $0\leq k\leq 3$, and $A\in \mathcal{A}$ be a $\mu$-ASD connection (i.e. $F^+_A = \mu(F^-_A)$).
Let $\mathrm{Ker}d_A^{*,W}\cap L^{2,W}_k$ be the space of $b\in L^{2,W}_k(X,\Lambda^1(\ad E))$ satisfying
$d_A^{*, W}b = W^{-2}d_A^{*}(W^2b)=0$ as a distribution.
Then the following map is isomorphic:
\begin{equation} \label{eq: linearlization of Coulomb gauge condition}
 L^{2,W}_{k+1}(X,\Lambda^0(\ad E))\oplus (\mathrm{Ker}d_A^{*,W}\cap L^{2,W}_k)\to L^{2,W}_k(X, \Lambda^1(\ad E)),\quad
 (u, b)\mapsto -d_Au + b.
\end{equation}
\end{lemma}
\begin{proof}
The above (\ref{eq: linearlization of Coulomb gauge condition}) is a bounded linear map.
If $(u,b)\in L^{2,W}_{k+1}\oplus (\mathrm{Ker}d_A^{*,W}\cap L^{2,W}_k)$ satisfies $-d_Au+b=0$, then 
$\norm{b}_{L^{2,W}}^2=\langle b,d_Au\rangle_{L^{2,W}}=0$ by $d_A^{*,W}b=0$
Hence $b=d_Au=0$.
$d_Au=0$ implies that $|u|$ is constant. But $u\in L^{2,W}$. Hence $u=0$.
Therefore the map (\ref{eq: linearlization of Coulomb gauge condition}) is injective. 

Let $b\in L^{2,W}_k(X,\Lambda^1(\ad E))$.
By Lemma \ref{lemma: Hodge decomposition of Lambda^0 Lambda^+} (iii),
there exists $(u, \eta)\in L^{2,W}_{k+1}(X,(\Lambda^0\oplus\Lambda^+)(\ad E))$ and $\beta\in \mathrm{Ker}D_A$
satisfying 
\[ b=D_A'(u,\eta)+\beta = -d_A u + d_A^{*,W}(1+\mu^*)\eta+\beta .\]
Since $D_A\beta=0$, we have $d_A^{*,W}\beta=0$.
Since $A$ is $\mu$-ASD, we have $d_A^{*,W} d_A^{*,W}(1+\mu^*)\eta=0$.
Thus $d_A^{*,W}(d_A^{*,W}(1+\mu^*)\eta+\beta)=0$.
This argument shows that the map (\ref{eq: linearlization of Coulomb gauge condition}) is surjective and hence isomorphic.
\end{proof}
\begin{proposition} \label{prop: crucial infinite dimensionality of the obstruction space}
Suppose $b_+(Y)\geq 1$.
For any $A=A_0+a\in \mathcal{A}$, the space $\mathrm{Ker}D'_A\subset L_4^{2,W}(X, (\Lambda^0\oplus\Lambda^+)(\ad E))$ is 
infinite dimensional.
\end{proposition}
\begin{proof}
Suppose that $\mathrm{Ker} D'_A$ is finite dimensional.
Then there is a pre-compact open set $V\subset X$ such that 
for any non-zero $\rho\in \mathrm{Ker}D_A'$ we have $\rho|_V\not\equiv 0$.
Then
\[ \norm{\rho}_{L^{2,W}(X)}\leq \const \norm{\rho}_{L^2(V)} \quad (\rho\in \mathrm{Ker}D_A').\]
We want to prove the following:
There exists a positive constant $C$ depending on $\mu,\alpha,T,A_0,A$ such that 
for any $\rho\in L^{2,W}_1(X,(\Lambda^0\oplus\Lambda^+)(\ad E))$
\begin{equation} \label{eq: the estimate implied by the finite dimensional assumption}
 \norm{\rho}_{L^{2,W}(X)}\leq C(\norm{\rho}_{L^{2}(V)} + \norm{D_A'\rho}_{L^{2,W}(X)}).
\end{equation}
Let $\rho\in L^{2,W}_1(X,(\Lambda^0\oplus\Lambda^+)(\ad E))$, and 
$\rho=\rho_0+\rho_1$ be a decomposition such that $\rho_0\in \mathrm{Ker}D_A'$ and that
$\rho_1\in L^{2,W}_1(X)$ is $L^{2,W}$-orthogonal to $\mathrm{Ker}D_A'$. 
Then 
\begin{equation}\label{eq: rho, rho_0 and rho_1}
 \begin{split}
 \norm{\rho}_{L^{2,W}(X)}&\leq \norm{\rho_0}_{L^{2,W}(X)}+\norm{\rho_1}_{L^{2,W}(X)}
   \leq \const \norm{\rho_0}_{L^{2}(V)}+\norm{\rho_1}_{L^{2,W}(X)},\\
  &\leq \const(\norm{\rho}_{L^{2}(V)} + \norm{\rho_1}_{L^{2}(V)}) + \norm{\rho_1}_{L^{2,W}(X)}\\
  &\leq \const\norm{\rho}_{L^{2}(V)} + (1+\const)\norm{\rho_1}_{L^{2,W}(X)}.
 \end{split}
\end{equation}
From Lemma \ref{lemma: Hodge decomposition of Lambda^0 Lambda^+} (ii) 
\[ \norm{\rho_1}_{L^{2,W}(X)}\leq \const\norm{D_A'\rho_1}_{L^{2,W}(X)}=\const\norm{D_A'\rho}_{L^{2,W}(X)}.\]
From this and the above (\ref{eq: rho, rho_0 and rho_1}), 
we get (\ref{eq: the estimate implied by the finite dimensional assumption}).

Take $\varepsilon>0$ satisfying $2\varepsilon C<1$ ($C$ is a constant in 
(\ref{eq: the estimate implied by the finite dimensional assumption})).
For this $\varepsilon$, we can choose a positive integer $R$ such that $V' := q^{-1}(-(R+1/2)T,(R+1/2)T)$ contains
$V\cup \supp(\mu)\cup \supp(F_{A_0})$ and that 
for any $\rho\in (\Omega^0\oplus\Omega^+)(\ad E)$ we have (see (\ref{eq: D'_A and D'_{A_0}}))
\[ |D_{A_0}'\rho(x)-D_A'\rho(x)|\leq \varepsilon |\rho(x)| \quad (x\in X\setminus V').\]
From Lemma \ref{lemma: techniacl lemma for infinite dimensionality of harmonic self-dual forms}, 
there is $\eta\in L^{2,W}_1(X, \Lambda^+(\ad E))$ such that 
$\eta=0$ over $V'$ and $\norm{D'_{A_0}\eta}_{L^{2,W}(X)} < \varepsilon \norm{\eta}_{L^{2,W}(X)}$.
(Here $D_{A_0}'\eta := D_{A_0}'(0,\eta)$.)
Then $\norm{D_A'\eta}_{L^{2,W}(X)} < 2\varepsilon\norm{\eta}_{L^{2,W}(X)}$.
But the above (\ref{eq: the estimate implied by the finite dimensional assumption}) implies 
\[ \norm{\eta}_{L^{2,W}(X)}\leq C\norm{D_A'\eta}_{L^{2,W}(X)} < 2C\varepsilon\norm{\eta}_{L^{2,W}(X)}.\]
Since we choose $2C\varepsilon<1$, this is a contradiction.
\end{proof}
Let $\mathrm{Ker}(d_A^{*,W}(1+\mu^*))$ be the space of $\eta\in L^{2,W}(X, \Lambda^+(\ad E))$ satisfying 
$d_A^{*,W}(1+\mu^*)\eta=0$ as a distribution.
If $\eta\in \mathrm{Ker}(d_A^{*,W}(1+\mu^*))$, then $D_A'(0, \eta)=0$.
Hence $\eta\in L^{2,W}_4(X)$ by Lemma \ref{lemma: elliptic estimate for D'_A}.
The space $\mathrm{Ker}(d_A^{*,W}(1+\mu^*))$
is closed in $L^{2,W}(X, \Lambda^+(\ad E))$, and hence 
it is closed in $L^{2,W}_k(X, \Lambda^+(\ad E))$ for all $0\leq k\leq 4$.
The following proposition is the conclusion of this section.
\begin{proposition} \label{prop: conclusion of the linear theory}
Suppose that $A\in \mathcal{A}$ is a $\mu$-ASD connection.

\noindent 
(i) Let $(u,\eta)\in L^{2,W}(X, (\Lambda^0\oplus \Lambda^+)(\ad E))$.
We have $D_A'(u,\eta)=0$ if and only if $u=0$ and $d_A^{*,W}(1+\mu^*)\eta=0$.
Hence $\mathrm{Ker}D_A' = \mathrm{Ker}(d_A^{*,W}(1+\mu^*))$.
Moreover if $b_+(Y)\geq 1$ then the space $\mathrm{Ker}(d_A^{*,W}(1+\mu^*))$ is infinite dimensional.

\noindent 
(ii) Let $0\leq k\leq 3$. 
Let $\mathrm{Ker} d_A^{*,W} \cap L^{2,W}_{k+1}$ be the space of $b\in L^{2,W}_{k+1}(X, \Lambda^1(\ad E))$
satisfying $d_A^{*,W}b=0$.
Then the space $(d_A^+-\mu d_A^-)(\mathrm{Ker}d_A^{*,W} \cap L^{2,W}_{k+1})$ is closed in $L^{2,W}_k(X, \Lambda^+(\ad E))$, 
and we have the following $L^{2,W}$-orthogonal decomposition:
\begin{equation} \label{eq: the conclusion: Hodge decomposition for Omega^+(ad E)}
 L^{2,W}_k(X, \Lambda^+(\ad E)) = \mathrm{Ker}(d_A^{*,W}(1+\mu^*))\oplus 
  (d_A^+-\mu d_A^-)(\mathrm{Ker}d_A^{*,W} \cap L^{2,W}_{k+1}).
\end{equation}
\end{proposition}
\begin{proof}
(i) Suppose $D_A'(u,\eta) = -d_A u + d_A^{*,W}(1+\mu^*)\eta =0$. 
Then $(u,\eta)\in L^{2,W}_4(X)$ by Lemma \ref{lemma: elliptic estimate for D'_A}.
Since $A$ is $\mu$-ASD,
$d_Au$ and $d_A^{*,W}(1+\mu^*)\eta$ are $L^{2,W}$-orthogonal to each other.
Hence $d_Au=d_A^{*,W}(1+\mu^*)\eta=0$.
Then $u=0$ and $d_A^{*,W}(1+\mu^*)\eta=0$.
Therefore $\mathrm{Ker}D_A' = \mathrm{Ker}(d_A^{*,W}(1+\mu^*))$.
If $b_+(Y)\geq 1$, then $\mathrm{Ker}(d_A^{*,W}(1+\mu^*)) = \mathrm{Ker}D_A'$ is infinite dimensional 
by Proposition \ref{prop: crucial infinite dimensionality of the obstruction space}.

(ii) By Lemma \ref{lemma: Hodge decomposition of Lambda^0 Lambda^+} (i),
$\eta\in L^{2,W}_k(X,\Lambda^+(\ad E))$ is $L^{2,W}$-orthogonal to $\mathrm{Ker}(d_A^{*,W}(1+\mu^*))$
if and only if there exists $b\in L^{2,W}_{k+1}(X, \Lambda^1(\ad E))$ satisfying 
$(0, \eta)=D_Ab$ (i.e. $d_A^{*,W}b=0$ and $(d_A^+-\mu d_A^-)b=\eta$).
This shows that $(d_A^+-\mu d_A^-)(\mathrm{Ker}d_A^{*,W}\cap L^{2,W}_{k+1})$ is closed in $L^{2,W}_k(X, \Lambda^+(\ad E))$
and that we have the decomposition (\ref{eq: the conclusion: Hodge decomposition for Omega^+(ad E)})
(the factors of the decomposition are $L^{2,W}$-orthogonal to each other). 
\end{proof}

\section{Non-existence of reducible instantons} \label{section: non-existence of reducible instantons}
\begin{lemma} \label{lemma: quantitative version of H^2 of the tube is zero}
Let $I\subset \mathbb{R}$ be an open interval.
Let $\omega$ be a smooth anti-self-dual $2$-form on $I\times S^3$ satisfying $d\omega=0$.
Then there exists a smooth $1$-form $a$ on $I\times S^3$ satisfying $da=\omega$ and 
$\norm{a}_{L^2(I\times S^3)} \leq (1/\sqrt{8})\norm{\omega}_{L^2(I\times S^3)}$.
\end{lemma}
\begin{proof}
Since $\omega$ is ASD, it can be written as:
\[ \omega = dt\wedge \phi -*_3\phi,\]
where $\phi\in \Gamma(I\times S^3,\Lambda^1_{S^3})$ (cf. Section \ref{subsection: preliminary estmates over the tube}).
Then $d\omega=0$ is equivalent to 
\[ \frac{\partial \phi}{\partial t}=-*_3d_3\phi, \quad d_3^*\phi =0.\]
Let $\mathrm{Ker} (d_3^*)\subset \Omega^1_{S^3}$ be the space of co-closed $1$-forms in $S^3$, and 
consider the operator $*_3d_3:\mathrm{Ker}(d_3^*)\to \mathrm{Ker}(d_3^*)$.
This is an isomorphism by $H^1_{dR}(S^3)=0$, and its inverse is given by 
$*_3d_3\Delta_3^{-1}:\mathrm{Ker}(d_3^*) \to \mathrm{Ker}(d_3^*)$. 
We set $a:= -*_3d_3\Delta_3^{-1}\phi\in \Gamma(I\times S^3,\Lambda^1_{S^3})$.
$a$ satisfies $d_3^*a=0$ and 
$*_3d_3a=-\phi$. Then
\[ *_3d_3\left(\frac{\partial a}{\partial t}\right) = -\frac{\partial \phi}{\partial t} = *_3d_3\phi.\]
Since $\partial a/\partial t$ and $\phi$ are both contained in $\mathrm{Ker}(d_3^*)$, 
we have $\partial a/\partial t= \phi$.
Then we have $da=\omega$.
Moreover (Corollary \ref{cor: estimate relating to the first eigenvelue of Laplacian on 1-forms} (i))
\[ \int_{\{t\}\times S^3}|\phi|^2d\vol_3 = \int_{\{t\}\times S^3}|d_3a|^2d\vol_3
    \geq 4 \int_{\{t\}\times S^3}|a|^2 d\vol_3.\]
Since $|\omega|^2 = 2|\phi|^2$, we get 
$\norm{\omega}_{L^2(I\times S^3)}\geq \sqrt{8}\norm{a}_{L^2(I\times S^3)}$.
\end{proof}
Let $\mu:\Lambda^-\to \Lambda^+$ be a compact-supported smooth bundle map satisfying 
$|\mu_x|<1$ for all $x\in X$. 
A $2$-form $\omega$ on $X$ is said to be $\mu$-ASD if it satisfies 
$\omega^+=\mu(\omega^-)$ where $\omega^+$ and $\omega^-$ are the self-dual and anti-self-dual 
parts of $\omega$ with respect to the periodic metric $g_0$.
$\omega$ is $\mu$-ASD if and only if $\omega$ is ASD with respect to the
conformal structure corresponding to $\mu$.
(See Corollary \ref{cor: description of conformal structures, manifold case}.)
\begin{proposition} \label{proposition: non-existence of closed L^2 asd forms}
Suppose $b_-(Y)=0$.
If $\omega$ is a smooth $\mu$-ASD $2$-form
on $X$ satisfying $d\omega=0$ and $\norm{\omega}_{L^2(X)}<\infty$,
then $\omega=0$.
(Indeed, if $\omega\in L^2(X,\Lambda^2)$ is $\mu$-ASD and satisfies $d\omega=0$ as a distribution, then 
$\omega$ is smooth by the elliptic regularity. Hence the assumption of the smoothness of $\omega$ can be weakened.)
\end{proposition}
\begin{proof}
Suppose $\omega\neq 0$. We can assume $\norm{\omega}_{L^2(X)}=1$.
We have $\int_X \omega\wedge\omega = \int_X (|\mu(\omega^-)|^2 -|\omega^-|^2)d\vol <0$.
So we can take $\delta>0$ so that $\int_X\omega\wedge\omega < -\delta$.
Let $\varepsilon>0$ be a positive number satisfying 
\begin{equation} \label{eq: choice of varepsilon in the proof of non-existence of closed L^2 asd forms}
(2+\varepsilon)\varepsilon \leq \delta/2.
\end{equation}
Let $N>0$ be a large integer such that $U:=q^{-1}(-NT, NT)$ satisfies $U\supset \supp (\mu)$ and
\begin{equation} \label{eq: norm of omega over X setminus U}
 \frac{2T-3}{T-2}\norm{\omega}_{L^2(X\setminus U)}\leq \varepsilon.
\end{equation}
(Recall $T>2$.)
Set $V := q^{-1}(-(N+1)T+1, -NT-1)\sqcup q^{-1}(NT+1, (N+1)T-1)$.
$V$ is isometric to the disjoint union of the two copies of $(1,T-1)\times S^3$, and we have $V\subset X\setminus U$.
From Lemma \ref{lemma: quantitative version of H^2 of the tube is zero}, there exists 
a $1$-form $a$ on $V$ satisfying $da=\omega$ and 
$\norm{a}_{L^2(V)}\leq (1/\sqrt{8})\norm{\omega}_{L^2(V)}$.
Let $\beta$ be a smooth function on $X$ such that $0\leq \beta\leq 1$, 
$\supp (d\beta)\subset V$, $\beta=0$ over $|t|\geq (N+1)T-1$, $\beta=1$ over $|t|\leq NT+1$
and $|d\beta|\leq 2/(T-2)$. (Here $t=q(x)$.)
We define a compact-supported $2$-form $\omega'$ by 
\begin{equation*}
\omega' := \begin{cases}
            \omega \quad &\text{on $|t|\leq NT+1$}\\
            d(\beta a)\quad &\text{on $V$}\\
            0 \quad &\text{on $|t|\geq (N+1)T-1$}.
           \end{cases}
\end{equation*}
$\omega'$ is a closed $2$-form ($d\omega'=0$).
\begin{equation*}
 \norm{d(\beta a)}_{L^2(V)}\leq \frac{2}{T-2}\norm{a}_{L^2(V)} + \norm{\omega}_{L^2(V)}
                           \leq \frac{1}{T-2}\norm{\omega}_{L^2(V)} + \norm{\omega}_{L^2(V)}
                           \leq \frac{T-1}{T-2}\norm{\omega}_{L^2(V)}.
\end{equation*}
Then, by (\ref{eq: norm of omega over X setminus U}), 
$\norm{\omega'-\omega}_{L^2(X)}\leq \frac{2T-3}{T-2}\norm{\omega}_{L^2(X\setminus U)} \leq \varepsilon$ and 
$\norm{\omega'}_{L^2(X)}\leq 1+\varepsilon$.
\[ \left|\int_X \omega\wedge\omega -\int_X\omega'\wedge\omega'\right| 
   \leq (\norm{\omega}_{L^2(X)}+\norm{\omega'}_{L^2(X)})\norm{\omega-\omega'}_{L^2(X)}
   \leq (2+\varepsilon)\varepsilon \leq \delta/2.\]
Here we have used (\ref{eq: choice of varepsilon in the proof of non-existence of closed L^2 asd forms}).
Since we have $\int_X\omega\wedge\omega < -\delta$,
\[ \int_X \omega'\wedge \omega' \leq -\delta/2.\]
On the other hand, since $\omega'$ is closed and compact-supported 
($\supp (\omega') \subset q^{-1}(-(N+1)T+1, (N+1)T-1)$), $\omega'$ can be 
considered as a closed $2$-form defined on $Y^{\sharp (2N+1)}$ (the connected sum of the
$(2N+1)$-copies of $Y$).
Since $b_-(Y)=0$, the intersection form of $Y^{\sharp (2N+1)}$ is positive definite.
Hence 
\[ 0\leq  \int_{Y^{\sharp (2N+1)}} \omega'\wedge\omega' = \int_X \omega'\wedge \omega' \leq -\delta/2.\]
Here $\delta$ is positive. This is a contradiction.
\end{proof}
Recall that $E=X\times SU(2)$ is the product principal $SU(2)$-bundle over $X$.
\begin{corollary} \label{corollary: non-existence of reducible instantons}
Suppose $b_-(Y)=0$.
If $A$ is a reducible $\mu$-ASD connection on $E$ satisfying 
\[ \int_X |F_A|^2d\vol < +\infty  , \]
then $A$ is flat.
\end{corollary}

\section{Moduli theory} \label{section: moduli theory}
\subsection{Sard-Smale's theorem} \label{subsection: Sard-Smale's theorem}
In this subsection we review a variant of Sard-Smale's theorem \cite{Smale} which will be used later.
Let $M_1$, $M_2$, $M_3$ be Banach manifolds. 
We assume that they are all second countable.
Let $f:M_1\times M_2\to M_3$ be a $\mathcal{C}^\infty$-map.
Let $(x_0, y_0)\in M_1\times M_2$ and set $z_0 := f(x_0,y_0)\in M_3$.
Suppose that the following two conditions hold.

\noindent
(i) The derivative $df_{(x_0,y_0)}:T_{x_0}M_1\oplus T_{y_0}M_2\to T_{z_0}M_3$ is surjective.

\noindent 
(ii) The partial derivative $d_1f_{(x_0,y_0)}:T_{x_0}M_1 \to T_{z_0}M_3$ with respect to $M_1$-direction
 is a Fredholm operator with 
$\dim\mathrm{Ker} (d_1f_{(x_0,y_0)}) < \dim \mathrm{Coker} (d_1f_{(x_0,y_0)})$.

Under these conditions we want to prove the following proposition.
(Recall that a subset of a topological space is said to be of first category if it is a 
countable union of nowhere-dense subsets.) 
\begin{proposition} \label{prop: a version of Sard-Smale's theorem}
There exists an open neighborhood $U\times U'\subset M_1\times M_2$ of $(x_0,y_0)$ such that 
the set $\{y\in U'|\, \exists x\in U: f(x,y) =z_0\}$ is of first category in $M_2$.
\end{proposition}
I believe that this is a standard result. But for the completeness of the argument we will give its brief proof below.
\begin{lemma} \label{lemma: the existence of right inverse of df_{(x_0,y_0)}}
There is a bounded linear map $Q: T_{z_0}M_3 \to T_{x_0}M_1\oplus T_{y_0}M_2$ which is a right inverse of $df_{(x_0,y_0)}$, i.e.
$df_{(x_0,y_0)}\circ Q=1$.
\end{lemma}
\begin{proof}
Set $D:=d_1f_{(x_0,y_0)}:T_{x_0}M_1 \to T_{z_0}M_3$.
Since $D$ is Fredholm, we have decompositions: 
$T_{x_0}M_1 = \mathrm{Ker} D\oplus V$ and $T_{z_0}M_3 =\mathrm{Im} D\oplus W$ where 
$V$ and $W$ are closed subspaces and moreover $W$ is finite dimensional.
The restriction $D|_{V} :V\to \mathrm{Im}D$ is an isomorphism.
Since $df_{(x_0,y_0)}$ is surjective and $W$ is finite dimensional, 
there is a bound linear map $T:W\to T_{x_0}M_1 \oplus T_{y_0}M_2$ satisfying 
$df_{(x_0,y_0)}\circ T=1$.
Then the map 
\[ Q:T_{z_0}M_3 = \mathrm{Im} D\oplus W \to T_{x_0}M_1\oplus T_{y_0}M_2, \quad (u,v)\mapsto (D|_V)^{-1}(u) + T(v),\]
gives a right inverse of $df_{(x_0,y_0)}$.
\end{proof}
By the implicit function theorem, there is an open neighborhood $U\times U' \subset M_1\times M_2$ of $(x_0,y_0)$
such that  
\[ M := \{(x,y)\in U\times U'|\, f(x,y)=z_0\} \]
is a smooth submanifold of $M_1\times M_2$, and that for any $(x,y)\in U\times U'$ the derivative 
$df_{(x,y)}:T_x M_1\oplus T_y M_2\to T_{f(x,y)}M_3$ is surjective.
Let $\pi:M\to M_2$ be the natural projection.

The set of Fredholm operators is open in the space of bounded operators, and the index is locally constant on it.
Hence we can choose $U$ and $U'$ so small that for any $(x,y)\in U\times U'$ the map 
$d_1f_{(x,y)}:T_x M_1\to T_{f(x,y)}M_3$ is Fredholm and satisfies 
$\dim \mathrm{Ker}(d_1f_{(x,y)}) < \dim \mathrm{Coker}(d_1f_{(x,y)})$.
For $(x,y)\in M$ we have
\[ T_{(x,y)}M = \{(u,v)\in T_x M_1\oplus T_y M_2|\, d_1f_{(x,y)}u + d_2f_{(x,y)}v=0\}.\]
Then it is easy to see that $\pi:M\to M_2$ is a Fredholm map with 
$\mathrm{Ker}(d\pi_{(x,y)}) \cong \mathrm{Ker}(d_1f_{(x,y)})$ and 
$\mathrm{Coker}(d\pi_{(x,y)})\cong \mathrm{Coker}(d_1f_{(x,y)})$ for $(x,y)\in M$.
(The maps $\mathrm{Ker}(d_1f_{(x,y)})\ni u\mapsto (u,0)\in \mathrm{Ker}(d\pi_{(x,y)})$ and 
$\mathrm{Coker}(d\pi_{(x,y)})\ni [v]\mapsto [d_2f_{(x,y)}(v)]\in \mathrm{Coker}(d_1f_{(x,y)})$
give isomorphisms.)
In particular  
$\mathrm{Index} (d\pi_{(x,y)}) < 0$ for $(x,y)\in M$.
Then a point $y\in M_2$ is regular for $\pi$ if and only if $\pi^{-1}(y)$ is empty.
We apply Sard-Smale's theorem to the map $\pi$ and conclude that 
$\pi(M)$ is of first category in $M_2$.
This proves Proposition \ref{prop: a version of Sard-Smale's theorem}.

\subsection{Review of Floer's function space}\label{subsection: Review of Floer's functional space}
Here we review a function space introduced by Floer \cite{Floer}.
Let $\vec{\tau} = (\tau_0, \tau_1, \tau_2,\cdots)$ be a sequence of positive real numbers indexed by $\mathbb{Z}_{\geq 0}$.
(We will choose a special $\vec{\tau}$ below.)
Let $\mathcal{C}^\infty(\mathbb{R}^n)$ be the set of all $\mathcal{C}^\infty$-functions in $\mathbb{R}^n$.
(We will need only the case $n=4$.)
For $f\in \mathcal{C}^\infty(\mathbb{R}^n)$ we set $|\nabla^k f(x)| := \max_{|\alpha|=k}|\partial^\alpha f(x)|$ $(k\geq 0)$
where $\alpha = (\alpha_1, \cdots, \alpha_n) \in \mathbb{Z}_{\geq 0}^n$ and $|\alpha| := \alpha_1+\cdots+\alpha_n$.
We define the norm $\norm{f}_{\vec{\tau}}$ by 
\[ \norm{f}_{\vec{\tau}} := \sum_{k\geq 0} \tau_k \sup_{x\in \mathbb{R}^n}|\nabla^k f(x)|. \]
We define $\mathcal{C}^{\vec{\tau}}(\mathbb{R}^n)$ as the set of all $f\in \mathcal{C}^\infty(\mathbb{R}^n)$ 
satisfying $\norm{f}_{\vec{\tau}}<\infty$.
$(\mathcal{C}^{\vec{\tau}}(\mathbb{R}^n), \norm{\cdot}_{\vec{\tau}})$ becomes a Banach space.
For an open set $U\subset \mathbb{R}^n$ we define $\mathcal{C}^{\vec{\tau}}_0(U)$ as the space of all 
$f\in \mathcal{C}_0^{\vec{\tau}}(\mathbb{R}^n)$ satisfying $f(x)=0$ for all $x\in \mathbb{R}^n\setminus U$.
$\mathcal{C}^{\vec{\tau}}_0(U)$ is a closed subspace in $\mathcal{C}^{\vec{\tau}}(\mathbb{R}^n)$.
\begin{lemma}  \label{lemma: separability of Floer's function space}
For any bounded open set $U\subset \mathbb{R}^n$,
$\mathcal{C}_0^{\vec{\tau}}(U)$ is separable.
\end{lemma}
\begin{proof}
Let $\mathcal{C}_0(\mathbb{R}^n)$ be the Banach space of all continuous functions $f$ in $\mathbb{R}^n$ which vanish
at infinity (i.e. for any $\varepsilon>0$ there is a compact set $K\subset \mathbb{R}^n$
such that $|f(x)| \leq \varepsilon$ for all $x\in \mathbb{R}^n\setminus K$).
Let $B$ be the set of all sequences $\vec{f} = (f_\alpha)_{\alpha\in \mathbb{Z}_{\geq 0}^n}$ with 
$f_\alpha \in \mathcal{C}_0(\mathbb{R}^n)$ satisfying 
\[ \norm{\vec{f}}_B := \sum_{k\geq 0} \tau_k \max_{|\alpha|=k} \norm{f_\alpha}_{\mathcal{C}_0(\mathbb{R}^n)} 
    <\infty , \quad (\norm{f_{\alpha}}_{\mathcal{C}_0(\mathbb{R}^n)} := \sup_{x\in \mathbb{R}^n}|f_{\alpha}(x)|).\]
$(B,\norm{\cdot}_B)$ is a Banach space.
Since $\mathcal{C}_0(\mathbb{R}^n)$ is separable, $B$ is also separable.
The map 
\[ \mathcal{C}^{\vec{\tau}}_0(U) \to B, \quad f\mapsto (\partial^\alpha f)_{\alpha\in \mathbb{Z}_{\geq 0}^n}, \]
is an isometric embedding.
(Note that $\partial^\alpha f$ vanishes at infinity because $f=0$ outside $U$ and $U$ is bounded.)
Hence $\mathcal{C}^{\vec{\tau}}_0(U)$ is separable.
\end{proof}
Let $\beta:\mathbb{R}\to \mathbb{R}$ be a $\mathcal{C}^\infty$-function satisfying 
$\beta(x)=0$ for $x\leq 1/3$ and $\beta(x)=1$ for $x\geq 2/3$.
We define positive numbers $a_k$ $(k\geq 0)$ by setting 
\[ a_k := \max_{x\in \mathbb{R}} |\beta^{(k)}(x)| + a_{k-1}  \quad (a_{-1}:=0).\]
Here $\beta^{(k)}$ is the $k$-th derivative of $\beta$.
We set $\tau_k := (a_k^n k^k)^{-1}$ $(k\geq 1)$ and $\tau_0:=1$.

For $0<\delta<L$, we define a $\mathcal{C}^\infty$-function $\beta_{\delta, L}:\mathbb{R}\to \mathbb{R}$ by 
\begin{equation*}
 \beta_{\delta,L}(x) := \begin{cases}
                           \beta(\frac{x+L}{\delta})  &(x\leq 0) \\
                           \beta(\frac{-x+L}{\delta}) &(x\geq 0).
                        \end{cases}
\end{equation*}
$\beta_{\delta, L}$ approximates the characteristic function of the interval $[-L, L]$ as $\delta\to 0$.
We have 
\begin{equation} \label{eq: upper bound on beta_{delta,L}}
|\beta_{\delta, L}^{(k)}(x)|\leq \delta^{-k}a_k \quad (k\geq 0).
\end{equation}
Note that the right-hand-side is independent of $L$.
For $y\in \mathbb{R}^n$ and $0<\delta<L$ we set 
\[ f_{y,\delta,L}(x) := \prod_{i=1}^n\beta_{\delta,L}(x_i-y_i).\]
$f_{y,\delta,L}$ is supported in the open cube 
$K_{y,L} := (y_1-L,y_1+L)\times \cdots \times (y_n-L,y_n+L)$, and 
$\lim_{\delta\to 0}f_{y,\delta,L} = 1_{K_{y,L}}$ (the characteristic function of 
$K_{y,L}$) in $L^r(\mathbb{R}^n)$ for $1\leq r< \infty$.
\begin{lemma}  \label{lemma: bunmp functions are contained in C^tau}
$f_{y,\delta,L}$ is contained in $\mathcal{C}^{\vec{\tau}}_0(K_{y,L})$.
\end{lemma}
\begin{proof}
For $\alpha = (\alpha_1, \cdots, \alpha_n)$,
$\partial^\alpha f_{y,\delta,L}(x) = \prod_{i=1}^n \beta^{(\alpha_i)}_{\delta,L}(x_i-y_i)$.
By using (\ref{eq: upper bound on beta_{delta,L}}),
\[ |\partial^\alpha f_{y,\delta,L}(x)|\leq \prod_{i=1}^n (\delta^{-\alpha_i}a_{\alpha_i})
                                      \leq \delta^{-|\alpha|} a_{|\alpha|}^n.\]
Hence $|\nabla^k f_{y,\delta,L}(x)| = \max_{|\alpha|=k}|\partial^\alpha f_{y,\delta,L}(x)|\leq \delta^{-k}a_k^n$.
Therefore 
\[ \sum_{k\geq 0}\tau_k \sup_{x\in \mathbb{R}^n}|\nabla^kf_{y,\delta,L}(x)| 
   \leq 1 + \sum_{k\geq 1}(a_k^n k^k)^{-1}\delta^{-k}a_k^n
   = 1 + \sum_{k\geq 1}(k\delta)^{-k} < \infty.\]
Thus $\norm{f_{y,\delta,L}}_{\vec{\tau}} < \infty$.
\end{proof}
\begin{lemma}  \label{lemma: C^tau is dense in L^r}
For any open set $U\subset \mathbb{R}^n$ and $1\leq r<\infty$, 
the space $\mathcal{C}^{\vec{\tau}}_0(U)$ is dense in $L^r(U)$.
\end{lemma}
\begin{proof}
It is enough to prove that for any $\varepsilon >0$ and any measurable set $E\subset U$ with $\vol(E)<\infty$
there exists $f\in \mathcal{C}^{\vec{\tau}}_0(U)$ satisfying 
$\norm{f-1_E}_{L^r(\mathbb{R}^n)} < \varepsilon$.

There is an open set $V\subset U$ satisfying $E\subset V$ and $\vol(V\setminus E) < (\varepsilon/4)^r$.
By Vitali's covering theorem, there are open cubes
$K_i = K_{y_i, L_i} \subset V$ $(i=1, 2, \cdots, N)$ such that $K_i\cap K_j=\emptyset$ $(i\neq j)$ and 
$\vol(E\setminus \bigsqcup_{i=1}^N K_i) < (\varepsilon/4)^r$.
Then 
\[ \norm{1_E -\sum_{i=1}^N 1_{K_i}}_{L^r} 
   \leq \left( \vol(E\setminus \bigsqcup_{i=1}^N K_i) + \vol(V\setminus E)\right)^{1/r} \leq \varepsilon/2.\]
From Lemma \ref{lemma: bunmp functions are contained in C^tau}, 
there are $f_i \in \mathcal{C}^{\vec{\tau}}_0(K_i)\subset \mathcal{C}_0^{\vec{\tau}}(U)$ $(i=1, \cdots,N)$ satisfying 
$\norm{1_{K_i}-f_i}_{L^r} < \varepsilon/2^{i+1}$.
Then 
\[ \norm{1_E-\sum_{i=1}^Nf_i}_{L^r} \leq \norm{1_E-\sum_{i=1}^N1_{K_i}}_{L^r} 
   + \sum_{i=1}^N \norm{1_{K_i} -f_i}_{L^r} < \varepsilon.\]
\end{proof}
Let us go back to our infinite connected sum space 
$X = Y^{\sharp \mathbb{Z}}$.
Take two non-empty pre-compact open sets $U$ and $V$ in $X$ such that 
$\overline{U}\subset V$ and $V$ is diffeomorphic to $\mathbb{R}^4$.
We fix a diffeomorphism between $V$ and $\mathbb{R}^4$ (i.e. a coordinate chart on $V$).
Moreover we fix bundle trivializations of $\Lambda^+$ and $\Lambda^-$ over $V$.
Here $\Lambda^+$ and $\Lambda^-$ are the vector bundles of self-dual and anti-self-dual $2$-forms 
with respect to $g_0$.
Then a bundle map $\mu: \Lambda^-|_V\to \Lambda^+|_V$ over $V$ can be identified with a matrix-valued function in $\mathbb{R}^4$
by using the coordinate chart on $V$ and the bundle trivializations of $\Lambda^+|_V$ and $\Lambda^-|_V$.
So we can consider its norm $\norm{\mu}_{\vec{\tau}}$.
We define a function space $\mathcal{C}^{\vec{\tau}}_0(U, \mathrm{Hom}(\Lambda^-, \Lambda^+))$
as the set of all $\mathcal{C}^\infty$-bundle maps $\mu: \Lambda^-|_V\to \Lambda^+|_V$ satisfying 
$\norm{\mu}_{\vec{\tau}}<\infty$ and $\mu_x =0$ for all $x\in V\setminus U$.
From Lemmas \ref{lemma: separability of Floer's function space} and \ref{lemma: C^tau is dense in L^r}, 
we get the following.
\begin{lemma} \label{lemma: separability and density of Floer's function space}
The Banach space $\mathcal{C}_0^{\vec{\tau}}(U, \mathrm{Hom}(\Lambda^-, \Lambda^+))$ is separable, and it is
dense in $L^r(U, \mathrm{Hom}(\Lambda^-, \Lambda^+))$ $(1\leq r<\infty)$.
\end{lemma}
Since $\mu\in \mathcal{C}^{\vec{\tau}}_0(U,\mathrm{Hom}(\Lambda^-,\Lambda^+))$ vanishes outside of $U$ and 
$\overline{U}\subset V$, $\mu$ can be smoothly extended all over $X$ by zero.
By this extension, we consider that all $\mu\in \mathcal{C}_0^{\vec{\tau}}(U, \mathrm{Hom}(\Lambda^-,\Lambda^+))$
are defined over $X$.

\subsection{Metric perturbation} \label{subsection: metric perturbation}
Recall that $X=Y^{\sharp \mathbb{Z}}$ is the infinite connected sum space with the periodic metric $g_0$
and the weight function $W = e^{\alpha |q(x)|'}$, 
and that $E=X\times SU(2)$ is the product principal $SU(2)$-bundle on $X$.
In this subsection we suppose that $0<\alpha<1$ and
the condition (\ref{eq: condition on T for the linear theory})
in Section \ref{section: linear theory} holds.
Therefore we can use the results proved in Section \ref{section: linear theory}.

Let $A_0$ be an adapted connection on $E$.
We define $\mathcal{A}=\mathcal{A}_{A_0}$ as the set of connections $A=A_0+a$ with 
$a\in L^{2,W}_3(X, \Lambda^1(\ad E))$
(Section \ref{subsection: the image of D_A is closed}).
Note that the definition of the Sobolev space $L^{2,W}_3(X, \Lambda^1(\ad E))$ uses
the connection $A_0$.
Let $\mathcal{C}\subset \mathcal{C}_0^{\vec{\tau}}(U, \mathrm{Hom}(\Lambda^-,\Lambda^+))$
be the set of all $\mu\in  \mathcal{C}_0^{\vec{\tau}}(U, \mathrm{Hom}(\Lambda^-,\Lambda^+))$
satisfying $|\mu_x| <1$ for all $x\in X$.
Here the norm $|\mu_x|$ is defined by using the metric $g_0$.
$\mathcal{C}$ is an open set in $\mathcal{C}^{\vec{\tau}}_0(U, \mathrm{Hom}(\Lambda^-,\Lambda^+))$.
Each $\mu\in \mathcal{C}$ defines a conformal structure which coincides with $[g_0]$ outside $U$
(see Corollary \ref{cor: description of conformal structures, manifold case}).
A connection $A$ on $E$ is said to be $\mu$-ASD if it satisfies $F_A^+ = \mu(F_A^-)$
where $F^+_A$ and $F^-_A$ are the self-dual and anti-self-dual parts of $F_A$ with respect to $g_0$.
\begin{lemma} \label{lemma: elementary properties of parametrized moduli space}
(i) For any $A\in \mathcal{A}$ we have 
$\int_X tr(F_A\wedge F_A) = \int_X tr(F_{A_0}\wedge F_{A_0})$.

\noindent
(ii) If $A_0$ is not equivalent to a flat connection as an adapted connection,
then any $A\in \mathcal{A}$ is not flat.

\noindent 
(iii) If $A_0$ is equivalent to a flat connection as an adapted connection and if 
$A\in \mathcal{A}$ is $\mu$-ASD for some $\mu\in \mathcal{C}$, then $A$ is flat.

\noindent 
(iv) If $\int_X tr(F_{A_0}\wedge F_{A_0}) < 0$, then for any $\mu\in \mathcal{C}$ there is no 
$\mu$-ASD connection in $\mathcal{A}$.
\end{lemma}
\begin{proof}
(i)
It is enough to prove that for any compact-supported smooth $a\in \Omega^1(\ad E)$ we have 
$\int_X tr (F(A_0+a)^2) = \int_X tr (F(A_0)^2)$.
Since we have $tr (F(A_0+a)^2) -tr (F(A_0)^2) = d \, tr(2a\wedge F(A_0) + a\wedge d_{A_0}a + \frac{2}{3}a^3)$,
it follows from Stokes' theorem.

\noindent
(ii) Since $A_0$ is not equivalent to the flat connection as an adapted connection,
the integral $\int_X tr(F(A_0)^2)$ is not equal to zero.
(See Proposition \ref{prop: classification of adapted connections}.)
Hence the result follows from (i).

\noindent 
(iii) 
If $A\in \mathcal{A}$ is $\mu$-ASD, then 
$tr(F_A^2) = (|F_A^-|^2 -|\mu(F_A^-)|^2) d\vol$ ($d\vol$ is the volume form with respect to $g_0$).
We have $|F_A^-|^2 -|\mu(F_A^-)|^2\geq 0$, and moreover if $F_A$ is not zero at $x\in X$ then 
$|F_A^-|^2 -|\mu(F_A^-)|^2> 0$ at $x\in X$.
If $A_0$ is equivalent to the flat connection, then $\int_X tr(F(A_0)^2) =0$.
Hence if $A\in \mathcal{A}$ is $\mu$-ASD then 
\[ \int_X (|F^-_A|^2 -|\mu(F_A^-)|^2) d\vol =0.\]
Therefore $F_A=0$ all over $X$.
We can prove (iv) by a similar argument.
\end{proof}
We define $\mathcal{M} \subset \mathcal{A}\times \mathcal{C}$ by 
\[ \mathcal{M} := \{(A, \mu)\in \mathcal{A}\times \mathcal{C}|\, \text{$A$ is $\mu$-ASD}\}.\]
Let $\pi :\mathcal{M}\to \mathcal{C}$ be the projection.
The main purpose of this subsection is to prove the following proposition
by using the metric perturbation technique originally due to 
Freed-Uhlenbeck \cite{Freed-Uhlenbeck}.
\begin{proposition} \label{prop: generic metric theorem}
Suppose that $b_+(Y)\geq 1$ and $b_-(Y)=0$ and that $A_0$ is not equivalent to a flat connection 
as an adapted connection.
Then $\pi(\moduli)$ is of first category in $\mathcal{C}$.
\end{proposition}
In the rest of this subsection we always assume that $b_+(Y)\geq 1$ and $b_-(Y)=0$ and 
that $A_0$ is not equivalent to a flat connection as an adapted connection.

Fix $(A, \mu)\in \mathcal{M}$.
Let $\mathrm{Ker}d_A^{*,W}\cap L^{2,W}_3$ be the space of $b\in L^{2,W}_3(X, \Lambda^1(\ad E))$
satisfying $d_A^{*,W}b=0$.
Let $\mathrm{Ker}D_A$ be the space of $b\in L^{2,W}_4(X, \Lambda^1(\ad E))$ satisfying 
$D_A b = -d_A^{*,W}b+(d_A^+-\mu d_A^-)b=0$, and 
$\mathrm{Ker}(d_A^{*,W}(1+\mu^*))$ be the space of $\eta\in L^{2, W}_4(X, \Lambda^+(\ad E))$ satisfying 
$d_A^{*,W}(1+\mu^*)\eta=0$.
$\mathrm{Ker}D_A$ is finite dimensional (Proposition \ref{prop: closedness of the image of the linearlization}), and 
$\mathrm{Ker}(d_A^{*,W}(1+\mu^*))$ is infinite dimensional (Proposition \ref{prop: conclusion of the linear theory}).
Hence we can take a finite dimensional sub-vector space $H\subset \mathrm{Ker}(d_A^{*,W}(1+\mu^*))$ satisfying 
$\dim H > \dim \mathrm{Ker} D_A$.
Let $H'\subset \mathrm{Ker}(d_A^{*,W}(1+\mu^*))$ be the $L^{2,W}$-orthogonal complement of $H$ in 
$\mathrm{Ker}(d_A^{*,W}(1+\mu^*))$.
Since $\mathrm{Ker}(d_A^{*,W}(1+\mu^*))$ is closed in $L^{2,W}(X, \Lambda^+(\ad E))$,
$H'$ is a closed subspace in $L^{2,W}(X, \Lambda^+(\ad E))$.

The spaces $(d_A^+-\mu d_A^-)(\mathrm{Ker}d_A^{*,W}\cap L^{2, W}_3)$, $H$ and $H'$ are
closed subspaces in $L^{2,W}_2(X, \Lambda^+(\ad E))$, and they are
$L^{2,W}$-orthogonal to each other (Proposition \ref{prop: conclusion of the linear theory} (ii)).
Moreover, from Proposition \ref{prop: conclusion of the linear theory} (ii), 
\begin{equation} \label{eq: decomposition of L^{2,W}_2 into three components}
 L^{2,W}_2(X, \Lambda^+(\ad E)) = (d_A^+-\mu d_A^-)(\mathrm{Ker}d_A^{*,W}\cap L^{2, W}_3)\oplus H \oplus H'.
\end{equation}
Let $\Pi: L^{2,W}_2(X, \Lambda^+(\ad E))\to (d_A^+-\mu d_A^-)(\mathrm{Ker}d_A^{*,W}\cap L^{2,W}_3)\oplus H$ be the 
projection with respect to this decomposition.
We define 
\begin{equation} \label{eq: truncated parametrized ASD equation}
 \begin{split}
 f: (\mathrm{Ker}d_A^{*,W}\cap L^{2, W}_3)\times \mathcal{C} &\to (d_A^+-\mu d_A^-)(\mathrm{Ker}d_A^{*,W}\cap L^{2,W}_3)\oplus H,\\
 (b, \nu) &\mapsto \Pi\{F^+(A+b) -\nu(F^-(A+b))\}.
 \end{split}
\end{equation}
We have $f(0, \mu)=0$.
The derivative of $f$ at $(0, \mu)$ is given by 
\begin{equation} \label{eq: the derivative of f in the metric perturbation}
 \begin{split}
 df_{(0,\mu)}: (\mathrm{Ker}d_A^{*,W}\cap L^{2,W}_3)\oplus \mathcal{C}_0^{\vec{\tau}}(U, \mathrm{Hom}(\Lambda^-,\Lambda^+))
  &\to  (d_A^+-\mu d_A^-)(\mathrm{Ker}d_A^{*,W}\cap L^{2, W}_3)\oplus H, \\
   (b, \nu)&\mapsto (d_A^+-\mu d_A^-)b -\Pi(\nu(F^-_A)). 
 \end{split}
\end{equation}
\begin{lemma}  \label{lemma: metric perturbation}
(i) The map (\ref{eq: the derivative of f in the metric perturbation}) is surjective.

\noindent 
(ii) The partial derivative $d_1 f_{(0,\nu)}:\mathrm{Ker}d_A^{*,W}\cap L^{2,W}_3
\to (d_A^+-\mu d_A^-)(\mathrm{Ker}d_A^{*,W}\cap L^{2,W}_3)\oplus H$,
$b\mapsto (d_A^+-\mu d_A^-)b$, with respect to $(\mathrm{Ker}d_A^{*,W}\cap L^{2, W}_3)$-direction is a Fredholm operator with 
its index $<0$.
\end{lemma}
\begin{proof}
The statement (ii) is obvious because $\mathrm{Ker}D_A$ and $H$ are both finite dimensional and satisfy 
$\dim \mathrm{Ker} D_A < \dim H$.

Next we will show (i) by using the argument of
Donaldson-Kronheimer \cite[p. 154]{Donaldson-Kronheimer}.
Let $\Pi_H:L^{2,W}_2(X, \Lambda^+(\ad E))\to H$ be projection to $H$ 
with respect to the decomposition (\ref{eq: decomposition of L^{2,W}_2 into three components}).
It is enough for the proof of (i) to show that the map 
$\Pi_H\circ df_{(0,\mu)}: (\mathrm{Ker}d_A^{*,W}\cap L^{2, W}_3)\oplus 
\mathcal{C}_0^{\vec{\tau}}(U, \mathrm{Hom}(\Lambda^-,\Lambda^+))\to H$ is surjective.
Here we have $\Pi_H\circ df_{(0,\mu)}(b,\nu)=-\Pi_H(\nu(F_A^-))$.

Suppose that it is not surjective.
Since $H$ is finite dimensional, this implies that there exists a non-zero $\eta\in H$ satisfying 
$\langle \eta, \nu(F_A^-)\rangle_{L^{2,W}}=0$ for all 
$\nu\in \mathcal{C}_0^{\vec{\tau}}(U, \mathrm{Hom}(\Lambda^-,\Lambda^+))$.
(Here $\langle \cdot, \cdot\rangle_{L^{2,W}}$ is the $L^{2,W}$-inner product.)
This is equivalent to 
$\langle F_A^-\cdot \eta, \nu\rangle_{L^{2,W}}=0$ for 
$\nu\in \mathcal{C}_0^{\vec{\tau}}(U, \mathrm{Hom}(\Lambda^-,\Lambda^+))$.
Here $F_A^-\cdot \eta \in \Gamma(\Lambda^-\otimes \Lambda^+)$ is the contraction of 
$F_A^-\otimes \eta \in \Gamma(\Lambda^-(\ad E)\otimes \Lambda^+(\ad E))$ by the 
inner product of $\ad E$, and we identify $\Lambda^-\otimes\Lambda^+$
with $\mathrm{Hom}(\Lambda^-,\Lambda^+)$ by the metric $g_0$.
Since $\mathcal{C}_0^{\vec{\tau}}(U, \mathrm{Hom}(\Lambda^-,\Lambda^+))$ is dense
in $L^2(U, \mathrm{Hom}(\Lambda^-,\Lambda^+))$ (Lemma \ref{lemma: separability and density of Floer's function space}),
the above means that $F_A^-\cdot \eta=0$ over $U$.
Then for every point $x\in U$,
the images of the maps
\[ (F_A^-)_x:(\Lambda^-)_x^*\to (\ad E)_x , \quad 
    \eta_x: (\Lambda^+)^*_x \to (\ad E)_x,\]
are orthogonal to each other.
Since the rank of $\ad E$ is equal to $\dim su(2) = 3$, this implies that 
$\min (\mathrm{rank}(F_A^-)_x, \mathrm{rank}(\eta_x))\leq 1$
for every $x\in U$.
Then we use the following sublemma. 
This is \cite[Lemma (4.3.25)]{Donaldson-Kronheimer}.
\begin{sublemma} \label{sublemma: metric perturbation and reducibility}
Let $\mathcal{O} \subset X$ be an non-empty open set.
Suppose that one of the following conditions (i), (ii) is satisfied.
Then $A$ is reducible over $X$.
 
\noindent 
(i) There is $\phi\in \Gamma(\mathcal{O}, \Lambda^-(\ad E))$ such that $\phi$ has rank $1$ over $\mathcal{O}$
(as a map from $(\Lambda^-)^*$ to $\ad E$) and $d_A(1+\mu)\phi=0$ over $\mathcal{O}$.

\noindent
(ii) There is $\phi\in \Gamma(\mathcal{O}, \Lambda^+(\ad E))$ such that $\phi$ has rank $1$ over $\mathcal{O}$ 
(as a map from $(\Lambda^+)^*$ to $\ad E$) and $d_A(1-\mu^*)\phi=0$ over $\mathcal{O}$.
\end{sublemma}
\begin{proof}
We assume the condition (i). The case (ii) can be proved in the same way.
By making $\mathcal{O}$ smaller, we can assume that 
$\phi= s\otimes \omega$ where $s\in \Gamma(\mathcal{O}, \ad E)$ and $\omega\in \Gamma(\mathcal{O},\Lambda^-)$
with $|s|=1$.
Here $\omega$ is not zero at any point of $\mathcal{O}$.
$d_A(1+\mu)\phi=d_A(s\otimes (1+\mu)\omega)=0$ implies
\[ d_As\wedge (1+\mu)\omega + s\otimes d(1+\mu)\omega=0.\]
Since $|s|=1$, we have $0=d(s,s)= 2(d_As, s)$.
From this and the above equation, we get
$d_As\wedge (1+\mu)\omega=0$.
Since $\omega\in \Omega^-$ and $\mu(\omega)\in \Omega^+$,
\[ |d_As\wedge \omega| = \frac{1}{\sqrt{2}}|d_As| |\omega|, \quad
   |d_As\wedge \mu(\omega)| = \frac{1}{\sqrt{2}}|d_As| |\mu(\omega)|.\]
Since $|\mu(\omega)| < |\omega|$, $d_As\wedge (1+\mu)\omega=0$ implies 
$d_As=0$.
This shows that $A$ is reducible over $\mathcal{O}$.
Since $X$ is simply-connected and $A$ is $\mu$-ASD, 
the unique continuation principle (\cite[Lemma (4.3.21)]{Donaldson-Kronheimer}) 
implies that $A$ is reducible over $X$.
\end{proof}
We have $d_A (1+\mu)F_A^- = d_AF_A=0$ and $d_A ((1-\mu^*)W^2\eta)=0$
since $\eta\in H\subset \mathrm{Ker}(d_A^{*,W}(1+\mu^*))$.
If $F^-_A$ is zero on some non-empty open set, then $A$ is flat on it.
Then the unique continuation principle (\cite[pp. 150-152]{Donaldson-Kronheimer}, \cite{Agmon-Nirenberg},
\cite[p. 248, Remark 3]{Aronszajn}) implies that $A$ is flat all over $X$.
But this contradicts Lemma \ref{lemma: elementary properties of parametrized moduli space} (ii) 
because $A_0$ is not equivalent to a flat connection as an adapted connection.
Therefore $F_A^-$ cannot vanish on any non-empty open set. 
The unique continuation principle also
implies that $\eta$ cannot vanish on any non-empty open set.
(Note that $(1-\mu^*)W^2\eta$ is self-dual with respect to the 
conformal structure corresponding to $\mu$.)

Since we have $\min(\mathrm{rank}(F_A^-)_x, \mathrm{rank}(\eta_x))\leq 1$ for every $x\in U$,
there is a non-empty open set $\mathcal{O}\subset U$ such that one of 
$F_A^-$, $\eta$ has rank $1$ over $\mathcal{O}$.
Then one of the conditions (i), (ii) in Sublemma \ref{sublemma: metric perturbation and reducibility}
is satisfied.
Thus $A$ is reducible on $X$.
Then, from Corollary \ref{corollary: non-existence of reducible instantons},
$A$ is flat over $X$.
But this contradicts Lemma \ref{lemma: elementary properties of parametrized moduli space} (ii).
\end{proof}
$\mathrm{Ker}d_A^{*,W}\cap  L^{2,W}_3$ and 
$\mathcal{C}\subset \mathcal{C}^{\vec{\tau}}(U, \mathrm{Hom}(\Lambda^-,\Lambda^+))$ are both separable 
(see Lemma \ref{lemma: separability and density of Floer's function space}) and hence 
second countable. 
Therefore we can apply Proposition \ref{prop: a version of Sard-Smale's theorem} to the map $f$ in 
(\ref{eq: truncated parametrized ASD equation})
and conclude that there exists an open neighborhood 
$\mathcal{U}\times \mathcal{U}'$ of $(0,\mu)$ in $(\mathrm{Ker}d_A^{*,W}\cap L^{2,W}_3)\times \mathcal{C}$
such that the set 
$\{\nu\in \mathcal{U}'|\, \exists b\in \mathcal{U}: f(b, \nu)=0\}$ is of first category in $\mathcal{C}$.
\begin{lemma} \label{lemma: local generic metric theorem}
There exists an open neighborhood $\mathcal{V}$ of $(A,\mu)$ in $\mathcal{M}$ 
such that $\pi(\mathcal{V})$ is of first category in $\mathcal{C}$.
\end{lemma}
\begin{proof}
Consider the following map (Coulomb gauge):
\[ L^{2,W}_4(X,\Lambda^0(\ad E))\times (\mathrm{Ker}d_A^{*,W}\cap L^{2,W}_3)\to \mathcal{A},\quad 
   (u,b)\mapsto e^u(A+b).\]
The derivative of this map at $(0,0)$ is given by 
\[ L^{2,W}_4(X,\Lambda^0(\ad E))\oplus (\mathrm{Ker}d_A^{*,W}\cap L^{2,W}_3)\to L^{2,W}_3(X, \Lambda^1(\ad E)),\quad 
   (u, b)\mapsto -d_A u + b.\]
This is isomorphic (Lemma \ref{lemma: Coulom gauge}).
Therefore, by the inverse mapping theorem, there is an open neighborhood $\mathcal{W}$ of $A$ in $\mathcal{A}$ such that 
for any $B\in \mathcal{W}$ there are $u\in L^{2,W}_4(X, \Lambda^0(\ad E))$ and 
$b\in \mathcal{U} \subset (\mathrm{Ker}d_A^{*,W}\cap L^{2,W}_3)$ satisfying 
$B=e^u(A+b)$.
Set $\mathcal{V} := (\mathcal{W}\times \mathcal{U}') \cap \mathcal{M}$.
Then $\pi(\mathcal{V})$ is contained in the set $\{\nu\in \mathcal{U}'|\, \exists b\in \mathcal{U}: f(b, \nu)=0\}$,
which is of first category in $\mathcal{C}$.
\end{proof}
Since $\mathcal{M}\subset \mathcal{A}\times \mathcal{C}$ is second countable, 
Lemma \ref{lemma: local generic metric theorem} implies 
Proposition \ref{prop: generic metric theorem}.

\section{Proof of Theorem \ref{thm: main theorem}}\label{section: proof of the main theorem}
We will prove Theorem \ref{thm: main theorem} in this section.
So we assume $b_-(Y)=0$ and $b_+(Y)\geq 1$.
We fix $0< \alpha <1$. (For example, $\alpha=1/2$ will do.)
We choose a positive parameter $T$ so that
\[ T > \max\left(T_\alpha, T_{-\alpha}, \frac{4}{1-\alpha}\right).\]
This implies
\[ T > 4, \quad T\geq \max(T_\alpha, T_{-\alpha}), \quad 1-4/T > \alpha.\]
Recall that we assumed $T > 4$ in Section \ref{subsection: exponential decay} and 
$T\geq \max(T_\alpha, T_{-\alpha})$ in Sections \ref{section: linear theory} and \ref{subsection: metric perturbation}.
The condition $1-4/T>\alpha$ is related to 
Corollary \ref{cor: existence of adapted connection to which instanton converges exponentially}.
We will show that there is a complete Riemannian metric $g$ on $X$ satisfying the 
conditions (a) and (b) in Theorem \ref{thm: main theorem}.

Let $A(m)$ $(m\in \mathbb{Z})$ be adapted connections on $E$ introduced in 
Section \ref{subsection: classification of adapted connections}.
They satisfy $\int_X tr(F(A(m))^2) =8\pi^2 m$.
$A(0)$ is equivalent to a flat connection as an adapted connection. 
$\{A(m)|\, m\in \mathbb{Z}\}$ 
becomes a complete system of representatives of equivalence classes of adapted connections on $E$.
(See Proposition \ref{prop: classification of adapted connections}.)
We define $\mathcal{A}_m$ as the set of all connections $A(m) +a$ such that $a\in L^2_{3,loc}(X,\Lambda^1(\ad E))$  
satisfies $\nabla_{A(m)}^ka\in L^{2,W}$ for $0\leq k\leq 3$.
We set 
\[ \mathcal{M}_m := \{(A, \mu)\in \mathcal{A}_m\times \mathcal{C}|\, \text{$A$ is $\mu$-ASD}\}.\]
Here $\mathcal{C}$ is the space of 
$\mu\in \mathcal{C}_0^{\vec{\tau}}(U, \mathrm{Hom}(\Lambda^-,\Lambda^+))$ satisfying 
$|\mu_x| <1$ $(x\in X)$ as in Section \ref{subsection: metric perturbation}.
If $m<0$, then $\mathcal{M}_m$ is empty by Lemma \ref{lemma: elementary properties of parametrized moduli space} (iv).
$(A, \mu)\in \mathcal{M}_0$ if and only if $A$ is flat 
by Lemma \ref{lemma: elementary properties of parametrized moduli space} (iii).

Let $\pi_m:\mathcal{M}_m \to \mathcal{C}$ be the natural projection.
Then $\bigcup_{m\geq 1}\pi_m(\mathcal{M}_m)$ is of first category in $\mathcal{C}$
by Proposition \ref{prop: generic metric theorem}.
$\mathcal{C}$ is an open set in the Banach space 
$\mathcal{C}^{\vec{\tau}}_0(U, \mathrm{Hom}(\Lambda^-,\Lambda^+))$.
Thus, by Baire's category theorem, there exists 
$\mu\in \mathcal{C}\setminus \left(\bigcup_{m\geq 1}\pi_m(\mathcal{M}_m)\right)$.
Let $g$ be a Riemannian metric on $X$ whose conformal equivalence class corresponds to $\mu$.
(See Corollary \ref{cor: description of conformal structures, manifold case}.)
Since $\mu$ is zero outside $U$ (a pre-compact open set in $X$), 
we can choose $g$ so that it is equal to $g_0$ outside a compact set.
In particular it is a complete metric.

We want to prove that there is no non-flat instanton with respect to the metric $g$.
Suppose, on the contrary, that there exists a non-flat $g$-ASD connection $A$ on $E$ satisfying 
$\int_X |F_A|^2d\vol < \infty$.
Then by Corollary \ref{cor: existence of adapted connection to which instanton converges exponentially}
and the condition $1-4/T>\alpha$,
there is a gauge transformation $u:E\to E$ such that 
$u(A)$ is contained in some $\mathcal{A}_m$.
This means that $\mu\in \pi_m(\mathcal{M}_m)$.
Since $A$ is not flat, we have $m\geq 1$.
This contradicts the choice of $\mu$. 

We have completed all the proofs of Theorem \ref{thm: main theorem}.


\vspace{10mm}

\address{ Masaki Tsukamoto \endgraf
Department of Mathematics, Faculty of Science \endgraf
Kyoto University \endgraf
Kyoto 606-8502 \endgraf
Japan
}

\textit{E-mail address}: \texttt{tukamoto@math.kyoto-u.ac.jp}


\begin{thebibliography}{100}




\bibitem{Agmon-Nirenberg}
S. Agmon, L. Nirenberg,
Lower bounds and uniqueness theorems for solutions of 
differential equations in a Hilbert space,
Comm. Pure Appl. Math. \textbf{20} (1967) 207-229




\bibitem{Aronszajn}
N. Aronszajn,
A unique continuation theorem for solutions of elliptic partial differential 
equations or inequalities of second order,
J. Math. Pures Appl. \textbf{36} (1957) 235-249




\bibitem{Atiyah}
M.F. Atiyah,
Elliptic operators, discrete groups, and von Neumann algebras,
colloque analyse et topologie en l'honneur de Henri Cartan,
Ast\'{e}risque, \textbf{32-33} (1976)




\bibitem{Donaldson}
S.K. Donaldson,
 Floer homology groups in Yang-Mills theory,
 with the assistance of M. Furuta and D. Kotschick,
 Cambridge University Press, Cambridge (2002)




\bibitem{Donaldson-Kronheimer}
S.K. Donaldson, P.B. Kronheimer,
 The geometry of four-manifolds,
 Oxford University Press, New York (1990)


\bibitem{Donaldson-Sullivan}
S.K. Donaldson, D.P. Sullivan, 
Quasiconformal $4$-manifolds,
Acta. Math. \textbf{163} (1989) 181-252



\bibitem{Floer}
A. Floer,
The unregularized gradient flow of the symplectic action,
Comm. Pure Appl. Math. \textbf{41} (1988) 775-813







\bibitem{Freed-Uhlenbeck}
D.S. Freed, K.K. Uhlenbeck,
 Instantons and four-manifolds, 
 Second edition, Springer-Verlag, New York (1991)


\bibitem{Fukaya}
K. Fukaya,
Anti-self-dual equation on 4-manifolds with degenerate metric,
GAFA \textbf{8} (1998) 466-528


\bibitem{Gilbarg-Trudinger}
D. Gilbarg, N. S. Trudinger,
Elliptic partial differential equations of second order,
Reprint of the 1998 edition,
Classics in Mathematics,
Springer-Verlag, Berlin (2001)



\bibitem{Matsuo-Tsukamoto}
S. Matsuo, M. Tsukamoto,
Instanton approximation, periodic ASD connections, and mean dimension,
preprint, arXiv: 0909.1141



\bibitem{Roe}
J. Roe,
Elliptic operators, topology and asymptotic methods,
Pitman Research Notes in Mathematics Series, 179, Longman Scientific \& Technical, Halow;
copublished in the United States with John Wiley \& Sons, Inc., New York (1988)



\bibitem{Sakai}
T. Sakai,
Riemannian geometry, 
translated from 1992 Japanese original by the author,
Translations of Mathematical Monographs, 149.
American Mathematical Society (1996)






\bibitem{Smale}
S. Smale,
An infinite dimensional version of Sard's theorem,
Amer. J. Math. \textbf{87} (1965) 861-866




\bibitem{Taubes}
C.H. Taubes,
 Self-dual connections on 4-manifolds with indefinite intersection matrix,
J. Differential Geom. \textbf{19} (1984) 517-560




\bibitem{Taubes periodic}
C.H. Taubes,
Gauge theory on asymptotically periodic $4$-manifolds,
J. Differential Geom. \textbf{25} (1987) 363-430





\bibitem{Tsukamoto 1}
M. Tsukamoto,
Gluing an infinite number of instantons,
Nagoya Math. J. \textbf{192} (2008) 27-58



\bibitem{Tsukamoto 2}
M. Tsukamoto,
Gauge theory on infinite connected sum and mean dimension,
Math. Phys. Anal. Geom. \textbf{12} (2009) 325-380




 
\end{thebibliography}
\end{document}